\documentclass{amsart}[10]

\usepackage{color}

\usepackage{graphicx,graphics}

\newcommand{\MC}{{\mathcal {MC}}}

\newcommand{\g}{{\bf g}}
\newcommand{\R}{{\bf R}}
\newcommand{\Z}{{\bf Z}}
\newcommand{\N}{{\bf N}}
\newcommand{\C}{{\bf C}}
\newcommand{\Se}{{\bf {S}}}

\newcommand{\dis}{{\bf d}}
\newcommand{\ac}{{\bf a}}
\newcommand{\bc}{{\bf b}}
\newcommand{\FD}{{\mathcal D}}
\newcommand{\E}{{\mathcal E}}
\newcommand{\Proj}{\mathcal {P}}
\newcommand{\Idle}{\mathcal {I}}
\newcommand{\wh}{\widehat}
\newcommand{\wt}{\widetilde}
\newcommand{\ph}{\partial}

\newcommand{\Homeo}{\operatorname{Homeo}}

\newcommand{\Diff}{\operatorname{Diff}}

\newcommand{\Mod}{\operatorname{Mod}}

\newcommand{\inte}{\operatorname{int}}

\newcommand{\chara}{\operatorname{ch}}

\newtheorem{theorem}{Theorem}[section]
\newtheorem{definition}{Definition}[section]
\newtheorem{lemma}{Lemma}[section]

\newtheorem{proposition}{Proposition}[section]

\theoremstyle{remark}
\newtheorem*{remark}{Remark}

\begin{document}

\author{Vladimir Markovic and Dragomir \v Sari\' c}

\address{Department of Mathematics, University of Warwick,
Coventry, CV8 4AL, United Kingdom} \email{v.markovic@warwick.ac.uk}

\address{Department of Mathematics, Queens College of CUNY,
65.10 Kissena Blvd., Flushing, NY 1.167}
\email{Dragomir.Saric@qc.cuny.edu}

\title[Realization by homeomorphisms]{The mapping class group cannot be realized by homeomorphisms}


\maketitle

\begin{abstract} Let $M$ be a closed surface. By $\Homeo(M)$ we denote the group of orientation preserving homeomorphisms of $M$ and let
$\MC(M)$ denote the Mapping class group. In this paper we complete the proof of the conjecture of Thurston that says that for any closed surface $M$ of genus $\g \ge 2$, there is no homomorphic section $\E:\MC(M) \to \Homeo(M)$ of the standard projection map $\Proj:\Homeo(M) \to \MC(M)$.
\end{abstract}

\section{Introduction}
Let $M$ be a closed surface of genus $\g \ge 2$. If $\wt{f} \in \Homeo(M)$ is a homeomorphism, then $[\wt{f}]=f \in \MC(M)$ denotes the corresponding homotopy class. Denote by $\Proj:\Homeo(M)\to \MC(M)$, the  projection from the group of orientation preserving homeomorphisms $\Homeo(M)$ of $M$, onto the mapping class group $\MC(M)$, that is   for $\wt{f} \in \Homeo(M)$, set $\Proj(\wt{f})=[\wt{f}]$. One stimulating question is whether
there exists a homomorphism $\E:\MC(M) \to \Homeo(M)$, so that $\Proj \circ \E$ is the identity mapping on $\MC(M)$. Such a homomorphism represents a homomorphic section (from now we just say section) of the projection $\Proj$. The mapping class group of the torus can be represented by homeomorphisms, namely  the corresponding section $\E$ exists. In fact, it can be represented as the group of affine transformations  ${\bf {SL}}_2(\Z)$.
\vskip .1cm
Morita \cite{mor} showed that there is no such section $\E :\MC(M)\to \Diff(M)$, when $\g > 4$, where $\Diff(M)$ is the
group of diffeomorphisms of $M$. Markovic \cite{ma} showed that a section $\E :\MC(M)\to \Homeo(M)$ does not exists when $\g>5$. Very recently Franks and Handel \cite{f-h-2} showed that $\E :\MC(M)\to \Diff(M)$ does not exists when $\g \ge 3$. In \cite{c-c}  Cantat and Cerveau showed that there is no section  $\E :\MC(M)\to \Diff^{\omega}(M)$ for any $\g \ge 2$, where $\Diff^{\omega}(M)$ is the group of real analytic diffeomorphisms of $M$. In fact in \cite{f-h-2} and \cite{c-c} it is shown that such sections do not exist when $\MC(M)$ is replaced by a finite index subgroup of $\MC(M)$.
In this paper we settle the  general case, by showing that such $\E:\MC(M) \to \Homeo(M)$ does not exists, where $M$ is a closed surface of any genus $\g \ge 2$. This of course settles  the case of diffeomorphisms   in the genus two case as well. The main result of this paper is the following theorem.

\begin{theorem}\label{main}  Let $M$ be a closed surface of genus $\g\ge 2$. Let $\Proj:\Homeo(M)\to \MC(M)$ be the projection. Then
there is no homomorphism
$\E :\MC(M)\to \Homeo(M)$, so that $\Proj \circ \E$ is the identity mapping on $\MC(M)$.
\end{theorem}
The proof of this theorem  is based on analysing certain Artin type relations in $\MC(M)$, proved by Farb-Margalit in \cite{f-m}, and eventually obtaining a contradiction with the existence of a homomorphic section $\E:\MC(M) \to \Homeo(M)$. Our proof does not distinguish between surfaces of different genus, except that we treat surfaces of even and odd genus in a slightly different manner (the differences are cosmetic).
We use techniques from \cite{ma} but ultimately we need several new ideas and technical gadgets to prove this theorem.
\vskip .1cm
We state these important relations in $\MC(M)$ and recall the notion of upper semi-continuous decompositions and the minimal decomposition for subgroups of $\Homeo(M)$ in Section 2. In Section 3 we introduce the notion of the twist number that is associated to a homeomorphism of an annulus, and prove the main preliminary results about dynamics of homeomorphisms actions on annuli. In Section 4 we assume the existence of a homomorphic section $\E:\MC(M) \to \Homeo(M)$ and construct the minimal annulus (this is a certain topological annulus in $M$ where the argument takes place). In Section 5 we prove Theorem \ref{main}. The Artin type relations from Section 2 will be used toward the end of Section 5.

\section{Important relations in $\MC(M)$ and the minimal decomposition}

\subsection{Relations in the mapping class group}
Recall that for a simple closed curve $\alpha$ by $t_{\alpha} \in \MC(M)$ we denote the twist about $\alpha$.
Given two simple closed curves $\alpha$ and $\beta$ on $M$, let
$( [\alpha] ,[\beta] )$ denote the geometric intersection number between
their homotopy classes.
\vskip .1cm
First consider the case when $M$ is a closed surface of even genus
$\g \ge 2$. Then there exists a separating simple closed curve
$\gamma$ on $M$ such that $M \setminus \gamma$ has two components
each homeomorphic to a closed surface of genus $\g/2$ minus a disk
(see Figure \ref{genus-2} and Figure \ref{genus-4}). Let   $e \in \MC(M)$ be an involution (that is $e^{2}=id$)
which interchanges the two components of $M \setminus \gamma$ and such that $e([\gamma])=[\gamma]$. There are many such involutions and we fix
one of them once and for all.
\vskip .1cm

\begin{figure}
\centering
\includegraphics{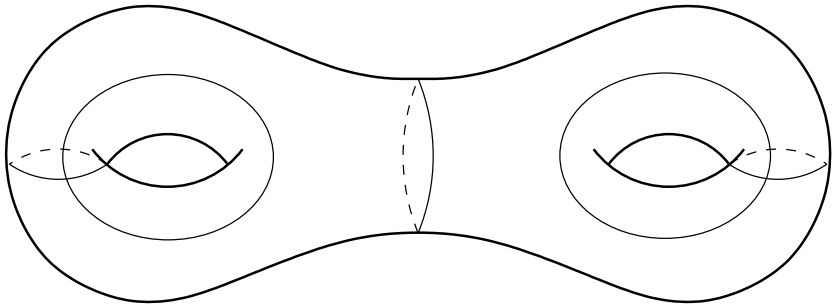}
\caption{The genus two case.}
\label{genus-2}
\end{figure}

\begin{figure}
\centering
\includegraphics{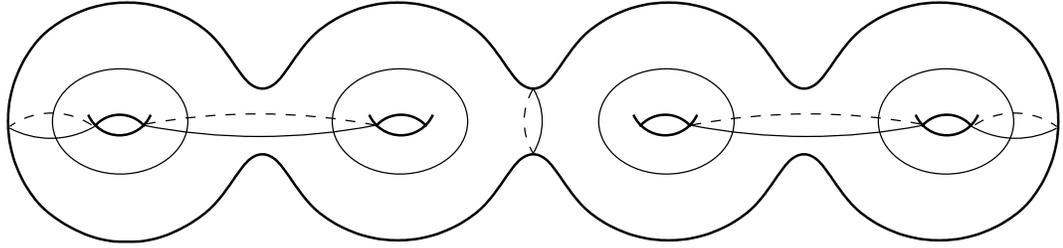}
\caption{The genus four case.} \label{genus-4}
\end{figure}

Let $\alpha_1,\ldots , \alpha_g$ be a $chain$ of simple closed curves on the left-hand side component of
$M \setminus \{\gamma\}$, namely  $( [\alpha_i],[\alpha_{i+1}] )=1$ for
$i=1,\dots ,g$, and $([\alpha_i],[\alpha_j])=0$ for $i,j=1,\ldots ,g$
with $|i-j|\ge 2$.
Let $\beta_i$ be curves such that $e([\alpha_i])=[\beta_i]$, for every $i$.
Then $\beta_1,\ldots ,\beta_g$ is a chain of simple closed curves on
the right-hand side component of $M \setminus \gamma$.
\vskip .1cm
The following Artin type relations are derived in the survey  paper by Farb-Margalit (see \cite{f-m})
\begin{equation}\label{even-relation}
(t_{\alpha_1}\circ \cdots \circ t_{\alpha_g})^{2g+2}=t_{\gamma}=
(t_{\beta_1} \circ \cdots \circ t_{\beta_g})^{2g+2}.
\end{equation}
\noindent

By $\ac_1$ we denote the set of curves $\alpha_{2i-1}$, $i=1,\ldots ,g/2$,  and by  $\ac_2$ we denote the set of curves
$\alpha_{2i}$, $i=1,\ldots ,g/2$.
Similarly, we denote by $\bc_1$ the set of curves $\beta_{2i-1}$, $i=1,2,\ldots ,g/2$, and denote by $\bc_2$ the set of curves
$\beta_{2i}$, $i=1,2,\dots ,g/2$. We have $e([\ac_j])=[\bc_j]$, where $j=1,2$.

\begin{figure}
  \centering
  \includegraphics{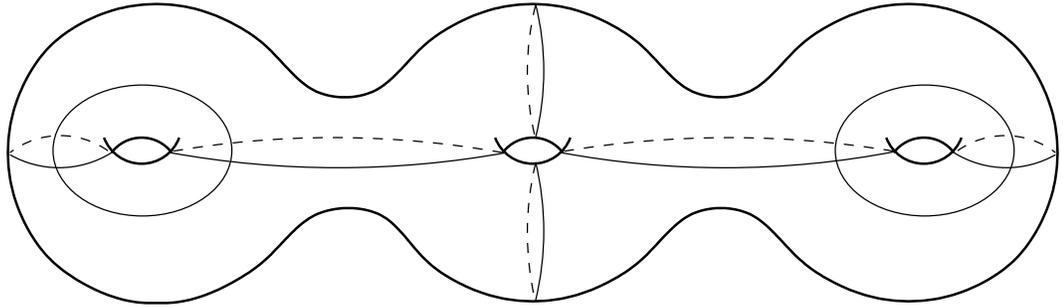}
  \caption{The genus three case.}\label{genus-3}
\end{figure}

\begin{figure}
\centering
\includegraphics{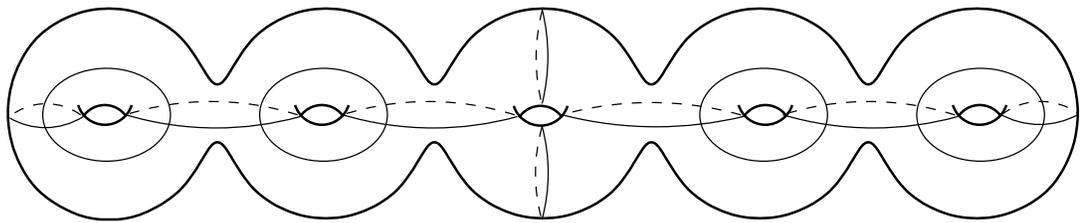}
\caption{The genus five case.}\label{genus-5}
\end{figure}

Next consider the case when $M$ is a closed surface of odd genus
$\g \ge 3$. Let $\gamma$ and $\gamma_1$ be two simple closed
non-intersecting curves on $M$ such that $M \setminus (\gamma \cup \gamma_1)$ has two components that are both homeomorphic to a twice holed
surface of genus $(g-1)/2$. Let $e \in \MC(M)$ be an involution
which interchanges the two components of $M \setminus (\gamma \cup \gamma_1)$, and such that
$e([\gamma])=[\gamma]$, and $e([\gamma_1])=[\gamma_1]$ (there are many such involutions and we fix one of them from now on).
Similarly as above (see Figure \ref{genus-3} and   Figure \ref{genus-5})  let $\alpha_1,\ldots ,\alpha_{g}$ be a chain of simple closed curves on
the left-hand side component of $M \setminus (\gamma \cup \gamma_1)$. Let $\beta_i$ be such that
$[\beta_i]=e([\alpha_i])$. Then $\beta_1,\ldots ,\beta_{g}$ is
a chain of simple closed curves on the right-hand side  component of
$M \setminus (\gamma\cup\gamma_1)$. We have the relations (see
\cite{f-m})
\begin{equation}\label{odd-relation}
(t_{\alpha_1} \circ \cdots  \circ t_{\alpha_g})^{g+1}=t_{\gamma}
t_{\gamma_1}=(t_{\beta_1} \circ \cdots \circ t_{\beta_g})^{g+1}.
\end{equation}

We denote by $\ac_1$ the set of curves $\alpha_{2i-1}$ for
$i=1,\ldots ,(g+1)/2$.  By $\ac_2$ the set of curves
$\alpha_{2i}$ for $i=1,\ldots ,(g-1)/2$ together with the curve $\gamma_1$ as
in Figure \ref{genus-3} and  Figure \ref{genus-5}. Similarly, we denote by $\bc_1$ the
set of curves $\beta_{2i-1}$ for $i=1,2,\ldots ,(g+1)/2$, and denote
by $\bc_2$ the set of curves $\beta_{2i}$ for $i=1,2,\dots,(g-1)/2$ together with the curve $\gamma_1$.
Then $e([\ac_i])=[\bc_i]$.
\vskip .1cm
The relations (\ref{even-relation}) and (\ref{odd-relation}) are the nontrivial relations we use in the paper. Beside this we will frequently use the following. If $\alpha$ and $\beta$ are simple closed curves such that $([\alpha],[\beta])=0$ then $t_{\alpha}$
and $t_{\beta}$ commute. Also, if
$f \in \MC(M)$ and $[\beta]=f([\alpha])$ then $f^{-1} \circ t_{\beta} \circ f=t_{\alpha}$.
\vskip .1cm
The strategy of our proof is that under the assumption that a section $\E:\MC(M) \to \Homeo(M)$ exists,  obtain a contradiction with the relations introduced above.

\subsection{The minimal decomposition}

Let $M$ be a closed surface of genus $\g \ge 2$.  First we recall several definition and results from \cite{ma} that are need in this paper.

\begin{definition}\label{upper-semi-continuous} Let $\Se$ be a collection of closed, connected
subsets of $M$. We say that $\Se$ is an upper semi-continuous decomposition of
$M$ if the following holds:
\begin{enumerate}
\item  If $S_1,S_2 \in \Se$, then $S_1 \cap S_2=\emptyset$.
\item If $S \in \Se$, then the set $M \setminus S$ does not contain a connected component that is simply connected.
\item We have

$$
M=\bigcup_{S \in \Se} S.
$$

\item   If $S_n \in \Se$, $n \in \N$, is a sequence that has the Hausdorff
limit $S_0$, then there exists $S \in \Se$ such that $S_0 \subset S$.

\end{enumerate}
\end{definition}

From now on by $\Se$ we always denote an upper semi-continuous decomposition of $M$.
Recall that a component  $S \in \Se$ is said to be acyclic if there is a simply connected
open set $U \subset M$ such that $S \subset U$ and $U \setminus S$ is homeomorphic to an annulus.
For every point $p \in M$, there exists a unique component in $\Se$ that contains $p$. We denote this component by $S_p \in \Se$.
The set of all points $p \in M$ such that the corresponding $S_p$ is acyclic is denoted by $M_{\Se}$.
The set  $M_{\Se}$ is  open  and every connected component of $M_{\Se}$
is a proper subsurface of $M$, which  represents the interior of a compact subsurface of $M$ with
finitely many ends.

\begin{definition}\label{admissible}
Let $\Se$ be an upper semi-continuous decomposition of $M$. Let $G$ be a
subgroup of $\Homeo(M)$. We say that $\Se$ is admissible for the group $G$ if the following
holds.

\begin{enumerate}
\item Each $\wh{f} \in G$ preserves setwise every component of $\Se$.

\item Let $S \in \Se$. Then every point, in every frontier component of
the surface $M \setminus S$, is a limit of points from $M \setminus S$
that belong to acyclic components of $\Se$ (note that not every point of $S$ need to be
in a frontier component of the subsurface $M \setminus S$).
\end{enumerate}

If $G$ is a cyclic group generated by a homeomorphism $\wh{f}:M \to M$ we
say that $\Se$ is an admissible decomposition for $\wh{f}$.
\end{definition}

For a generic homeomorphism $\wh{f}:M \to M$, the only admissible
decomposition is the trivial one, which is the one that contains only one
set, namely $M$ itself.

\begin{definition}\label{minimal-decomposition}
An admissible decomposition $\Se$ for a group $G$ will be called the
$minimal$ $decomposition$ for $G$
if $\Se$ is contained in every admissible decomposition for $G$.
\end{definition}

We have \cite{ma}

\begin{theorem}\label{unique-minimal-decomposition} Every group $G <\Homeo(M)$ has the unique minimal decomposition $\Se(G)$.
That is, if $\Se$ is an admissible decomposition for $G$ then for every $p \in M$ we have $S_p(G) \subset S_p$, where $S_p(G) \in \Se(G)$ and
$S_p \in \Se$.
\end{theorem}

Assuming that the section $\E:\MC(M) \to \Homeo(M)$ exists, we have the following lemma.

\begin{lemma}\label{acyclic-component}
Let $\ac$ denote a set of simple, closed, and  mutually disjoint curves on $M$, such that no two curves are homotopic and no curve is
homotopically  trivial. Let $G<\Homeo(M)$ be the group generated by $\E(t_{\alpha})$, where $\alpha \in \ac$, and let $\Se$ denote the minimal decomposition for $G$.  Suppose that $R$ is a connected component of the open  set obtained by removing the collection of curves $\ac$ from $M$.
If $R$  has  the negative Euler characteristic then there exists a unique component of $M_{\Se}$ homotopic to $R$.
\end{lemma}

\begin{proof} The set obtained by removing the collection of curves $\ac$ from $M$, is a finite and disjoint  union of surfaces with boundary. That is, each such surface is obtained by removing disjoint discs from closed surface. Assume that $R$ is one of these surface with boundary, and assume that $R$ has the negative Euler characteristic. We will show that there exists a component of $M_{\Se}$ that is homotopic to $R$.
\vskip .1cm
Let  $T \subset R$, be a subsurface of $R$, that is either homeomorphic  to a torus minus a disc, or to a sphere minus four disc. Since there exists an Anosov diffeomorphism on such $T$ it follows from Theorem 4.1 in \cite{ma} that we can find a surface $T_1 \subset M_{\Se}$ that is homotopic to $T$.

\begin{remark} In \cite{ma} this was proved in the case when $T$ is a torus minus a disc, but the proof is the same when $T$ is a sphere with four holes.
This lemma can be proved in a similar way by using the results from \cite{f-h-1}.
\end{remark}

Let $\alpha$ be a simple closed curve in $R$, that is not homotopic to an end of $R$, and that is not homotopically trivial. Since $R$  has the negative Euler characteristic we can find a surface $T \subset R$, that is either homeomorphic to a  torus minus a disc, or to a sphere minus four disc, and such that $T$ contains a simple closed curve $\alpha_1$ that is homotopic to $\alpha$. This shows that for every such curve $\alpha \subset R$ there exists a curve $\alpha_1$ that is homotopic to $\alpha$, and that belongs to $M_{\Se}$. This implies that there  exists a component of $M_{\Se}$ that contains a curve homotopic to any simple closed curve on $R$. Denote this component by $M_1$. We see that $M_1$ is homotopic to $R$. The uniqueness is obvious.
\end{proof}

\vskip .1cm
Recall the following definition.

\begin{definition}\label{triod} Let $K \subset M$ be a closed and  connected set. We say that $K$ is a triode if there exists a connected closed set $K_1 \subset K$, such that $K \setminus K_1$ has at least three connected components.
\end{definition}

The Moore's triode theorem says that any open subset of $M$ can contain at most countably many disjoint triodes (see \cite{mo}), \cite{mo-1}).

\begin{lemma}\label{triod-lemma} Let $F$ and $G$ be two groups of homeomorphisms of $M$ such that $\wt{f}$ commutes with $\wt{g}$ for every
$\wt{f} \in F$ and $\wt{g} \in G$. Denote by $\Se(F)$ and $\Se(G)$ the minimal decompositions that correspond to the groups $F$ and $G$ respectively.
Let $\Gamma$ be the group generated by the elements from $F$ and $G$ and let $\Se(\Gamma)$ be the corresponding minimal decomposition. By $M_{\Se(\Gamma)}$ we denote the set of all points that are contained in acyclic components of $\Se(\Gamma)$. Let $p \in M_{\Se(\Gamma)}$ and assume that $p$ does not belong to the interior of $S_p(\Gamma) \in \Se(\Gamma)$. Then at least one of the following two statements holds
\begin{itemize}
\item $\wt{f}(S_p(G))=S_p(G)$, for every $\wt{f} \in F$.
\item $\wt{g}(S_p(F) )=S_p(F)$, for every $\wt{g} \in G$.
\end{itemize}
(recall that  $S_p$ denotes the component from $\Se$ that contains the point $p$).
\end{lemma}

\begin{proof} Assume that $M_{\Se(\Gamma)}$ is non-empty (otherwise the lemma is trivial). Then by the minimality we have
$M_{\Se(\Gamma)} \subset (M_{\Se(F)} \cap M_{\Se(G)})$.  Let $p \in M_{\Se(\Gamma)}$ and assume that
\begin{equation}\label{triod-equation}
\wt{f}(S_p(G)) \ne S_p(G), \quad \text{and} \quad  \wt{g}(S_p(F)) \ne S_p(F),
\end{equation}
\noindent
for some $\wt{f} \in F$ and $\wt{g} \in G$. Since the groups $F$ and $G$ commute we have that $\wt{f}$ respects the minimal decomposition $\Se(G)$ and
that $\wt{g}$ respects the minimal decomposition $\Se(F)$.
This implies that no component $S_q(G) \in \Se(G)$ is  a subset of $S_p(F)$
(if $S_q(G) \subset S_p(F)$ then $\wt{g}(S_p(F))=S_p(F)$). Similarly we have that $S_p(F)$ is not a subset of $S_p(G)$.
Therefore we can find a point $p_1 \in S_p(F)$ such that $p_1$ does not belong to $S_p(G)$. We have that $S_p(G)$ and $S_{p_{1}}(G)$
are different components from $\Se(G)$ and therefore they are mutually disjoint.
\vskip .1cm
We already observed that $S_p(F)$ is not a subset of any $S_q(G)$. We show that $S_p(F)$ is not a subset of $S_p(G) \cup S_{p_{1}}(G)$. Consider the set $X=S_p(F) \setminus S_p(G)$. Then $X$ is a relatively open subset of $S_p(F)$. If
$S_p(F) \subset \big( S_p(G) \cup S_{p_{1}}(G) \big)$, and since $S_p(G) \cap S_{p_{1}}(G)=\emptyset$ we have that
$X=S_p(F) \cap S_{p_{1}}(G)$. This implies that $X$ is both relatively open and closed in $S_p(F)$ which shows that $X=S_p(F)$. This is a contradiction so we have  that $S_p(F)$ is not a subset of $S_p(G) \cup S_{p_{1}}(G)$.
\vskip .1cm
Therefore there exists $p_2 \in S_p(F)$  such that $S_{p_{2}}(G)$ is disjoint from $S_{p}(G)$ and $S_{p_{1}}(G)$. Consider now the component $S_p(\Gamma) \in \Se(\Gamma)$. By the minimality we have that
$$
Y=S_p(F)\cup S_{p}(G) \cup S_{p_{1}}(G) \cup S_{p_{2}}(G) \subset S_p(\Gamma).
$$
\noindent
On the other hand $Y$ is a connected closed set and $Y \setminus S_p(F)$ contains at least three connected components. This shows that $Y$ is a triode and we conclude that if for some $p\in M_{\Se(\Gamma)}$ we have that (\ref{triod-equation}) holds then $S_p(\Gamma)$ contains a triode.
By the Moore's theorem there could be at most countably many such components in $\Se(\Gamma)$. On the other hand the set of points $p \in M_{\Se(\Gamma)}$ such that at least one of the two conditions from the statement of this lemma holds is relatively closed in $M_{\Se(\Gamma)}$. In particular if $p_0 \in S_p(\Gamma) \in M_{\Se(\Gamma)}$ is such that $p_0$ does not belong to the interior of $S_p(\Gamma)$, then there exists a sequence of points $p_n \in  M_{\Se(\Gamma)}$ such that at least one of the two conditions  holds at $p_n$ and $p_n \to p_0$ (this follows from the definition of admissible decompositions). Then at least one of the two conditions holds at $p_0$. This proves the lemma.

\end{proof}

\section{The twist number and the analysis on the strip}
\subsection{The Translation number}
Let $\varphi:\R \to \R$, be a homeomorphism (orientation preserving) that commutes with the translation $T(x)=x+1$. By the classical result of Poincare, the limit

$$
\rho(\varphi)=\lim_{n\to\infty} {{1}\over{n}} \varphi^{n}(x),
$$
\noindent
exists, and it does not depend on $x \in \R$. The number $\rho(\varphi)$ is called the $translation$ $number$ of $\varphi$.
\vskip .1cm
Assume that $\varphi$ has a fixed point, that is $\varphi(x_0)=x_0$ for some  $x_0 \in \R$. Then $\varphi(x_0+k)=x_0+k$ for every $k \in \Z$. Therefore for every $x \in \R$, we have $|\varphi^{n}(x)-x| <1$ for every $n \in \N$, so in this case $\rho(\varphi)=0$. If $\varphi$ does not have a fixed point then either for every $x \in \R$ we have  $\varphi(x)>x$, or for every $x \in \R$ we have $\varphi(x)<x$. Suppose that  $\varphi(x)>x$, $x \in \R$. Since $\varphi$ commutes with the translation $T(x)$, by the compactness there exists $q>0$, such that $\varphi(x)>x+q$ for every $x \in \R$. We have $\varphi^{n}(x)>x+nq$, which shows that $\rho(\varphi)>0$. Similarly, if $\varphi(x)<x$, $x \in \R$, then $\rho(\varphi)<0$. We conclude that the translation number is zero if and only if $\varphi$ has a fixed point.
\begin{remark} Since $\varphi$ commutes with $T(x)$ we have that $\varphi$ is a lift of the circle homeomorphism $\varphi_1$. The rotation number of $\varphi_1$ is defined to be the translation number of $\varphi$ modulo $1$. It is not true that the rotation number of $\varphi_1$ is equal to zero if and only if $\varphi_1$ has a fixed point on the circle. However, the classical result says that   rotation number of $\varphi_1$ is a rational number if and only if some power of $\varphi_1$ has a fixed point on the circle.
\end{remark}

Note that $\rho(\varphi^{m})=m\rho(\varphi)$, and $\rho(T^{m} \circ \varphi)= \rho(\varphi \circ T^{m})=\rho(\varphi)+m$, for any $m \in \Z$.

\begin{proposition}\label{uniform-translation} Let $\varphi:\R \to \R$, be a homeomorphism (orientation preserving) that commutes with the translation $T(x)$. Then for every $x \in \R$, and for every $n \in \N$, we have
$$
|(\varphi^{n}(x)-x)-n\rho(\varphi)|<3.
$$
\end{proposition}
\begin{remark} This proposition shows that given a compact set $A \subset \R$, the sequence ${{1}\over{n}} \varphi^{n}(x)$ converges uniformly to $\rho(\varphi)$, for every $x \in A$, regardless of the choice of the homeomorphism $\varphi$.
\end{remark}

\begin{proof} Assume that for some $x_0 \in \R$, and some $n_0 \in \N$, we have
$$
|(\varphi^{n_{0}}(x)-x)-n_{0}\rho(\varphi)| \ge 3.
$$
\noindent
Then  either
$$
\varphi^{n_{0}}(x_0)-x_0>n_{0}\rho(\varphi) + 3,
$$
\noindent
or
$$
\varphi^{n_{0}}(x_0)-x_0<n_{0}\rho(\varphi)-3.
$$
\noindent
Consider the first case. Note that for every $x_1,x_2 \in \R$, with $|x_1-x_2|<1$, we have $|\varphi^{n}(x_1)-\varphi^{n}(x_2)|<1$, for every $n \in \N$.
This implies that for every $x \in \R$, we have
$$
\varphi^{n_{0}}(x)-x > (\varphi^{n_{0}}(x_0)-x_0)-2 \ge n_{0}\rho(\varphi)+3-2=n_{0}\rho(\varphi)+1.
$$
\noindent
This yields that for every $j \in \N$, by setting $x=\varphi^{(j-1)n_{0}}(x_0)$, the above inequality yields
$$
\varphi^{jn_{0}}(x_0)-\varphi^{(j-1)n_{0}}(x_0) > n_{0}\rho(\varphi)+1.
$$
\noindent
Let $k \in \N$, and $1 \le j \le k$. We sum up all the above inequalities for $1 \le j \le k$, and get
$$
\varphi^{kn_{0}}(x_0)-x_0 > kn_{0}\rho(\varphi)+k.
$$
\noindent
Letting $k \to \infty$, we obtain that
$$
\lim_{k \to \infty} {{1}\over{kn_{0}}}(\varphi^{kn_{0}}(x_0)-x_0) \ge \rho(\varphi)+{{1}\over{n_{0}} }>\rho(\varphi).
$$
\noindent
This is a contradiction. The second case is handled in the same way.
\end{proof}

\subsection{The twist number of an annulus homeomorphism}

In this section $z$ and $w$ represent complex variables in the complex plane $\C$. We have $Re(z)=x$ and $Im(z)=y$, that is $z=x+iy$.
Let $N(r)=\{ w \in \C: {{1}\over{r}}<|w|<r \}$, be the geometric annulus in the complex plane $\C$.
By $P(r)=\{x+iy=z \in \C: |y| < {{\log r}\over {2\pi}} \}$, we denote the geometric strip in $\C$. By $\overline{N(r)}$ and $\overline{P(r)}$, we denote the corresponding closures of $N(r)$ and $P(r)$ in the complex plane $\C$.

\begin{remark} We point out that $\overline{P(r)}$ is the closure of $P(r)$ in $\C$, that is $\infty$ does not belong to $\overline{P(r)}$.
\end{remark}

Let $\ph_0(N(r))=\{ w \in \C: |w|={{1}\over{r}}\}$, and  $\ph_1(N(r))=\{ w \in \C: |w|=r \}$. Similarly, set
$\ph_0(P(r))=\{ z \in \C : y=-i{{\log r}\over{2\pi}}\}$, and  $\ph_1(P(r))=\{ z \in \C: y=i{{\log r}\over{2\pi}}\}$.
The map given by $w=e^{-2\pi iz}$, is a holomorphic covering of the annulus $N(r)$ by the strip $P(r)$. Note that for every $r$, the covering group that acts on $P(r)$ is generated by the translation $T(z)=z+1$.
\vskip .1cm
By $\wt{e}:\overline{N(r)} \to \overline{N(r)}$, we always denote  a conformal involution that exchanges the two boundary circles. There are exactly two such involutions and they are given by  $\wt{e}(w)={{1}\over{w}}$, or $\wt{e}(w)={{e^{i\pi} }\over{w}}$.  By $\wh{e}:\overline{P(r)}\to \overline{P(r)}$ we denote a lift of $\wt{e}$ to $\overline{P(r)}$. For our purposes it is important to observe that every such $\wh{e}:\overline{P(r)}\to \overline{P(r)}$
is an isometry in the Euclidean metric.
\vskip .1cm
In the remainder of this section we fix $1<r_0$ and set $N(r_0)=N$ and $P(r_0)=P$.

\begin{definition}\label{twist-number} Let $\wh{f}:\overline{P} \to \overline{P}$, be a homemorphism that fixes setwise the boundary components of $P$, and that commutes with the translation $T(z)=z+1$. We define the $twist$  $number$  $\rho (\wh{f},P) \in \R$ as
$$
\rho (\wh{f},P)=\lim_{n\to\infty} {{1}\over{n}} \left(Re( \wh{f}^n(z_1))- Re(\wh{f}^n(z_0) ) \right),
$$
\noindent
where $z_0 \in \ph_0(P)$, and  $z_1 \in \ph_1(P)$.
\end{definition}

Note that the restriction of $\wh{f}$ to $\ph_0{P}$ is a homeomorphism, that commutes with the translation $T(z)$.
Therefore the sequence  ${{1}\over{n}} Re( \wh{f}^n(z_0))$ converges to the translation number of the corresponding homeomorphism of the real line. Moreover, this limit does not depend on the choice of $z_0 \in \ph_0(P)$. Similarly the sequence  ${{1}\over{n}} Re( \wh{f}^n(z_1))$ converges, and this limit does not depend on the choice of $z_1 \in \ph_1(P)$. This shows that $\rho (\wt{f},N)$ represents the difference in the translation numbers between the restriction of $\wh{f}$ to $\ph_1(P)$, and the restriction of  $\wh{f}$ to $\ph_0(P)$.

\begin{definition}
Let $\wt{f}:\overline{N} \to \overline{N}$, be a homemorphism that fixes setwise the boundary circles of $N$. We define the $twist$  $number$  $\rho (\wt{f},N) \in \R$ as follows. Let $\wh{f}:\overline{P} \to \overline{P}$ be a lift of $\wt{f}$. Then
$\rho (\wt{f},N)=\rho(\wh{f},P)$.
\end{definition}

Note that the assumption that $\wt{f}$ setwise preserves the boundary circles of $N$ implies that $\wh{f}$ setwise preserves the boundary lines of $P$.  Since any two lifts of $\wt{f}$ to $P$, differ by a translation, we see that $\rho (\wt{f},N)$ does not depend on the choice of the lift $\wh{f}$. For every $m \in\Z$ we have $\rho(\wt{f}^{m},N)=m\rho(\wt{f},N)$. Moreover, if $\wt{f}$ has at least one fixed point on both boundary circles then the twist number $\rho (\wt{f},N)$ is an integer. If $\wt{f}$ is homotopic to the geometric twist homeomorphism (modulo these fixed points) then
$\rho (\wt{f},N)=1$.

\begin{remark} If two homeomorphisms of $\overline{N}$ agree on the boundary of $N$, and if they are homotopic modulo the boundary, then the twists numbers agree. Moreover, our definition of the twist number of a homeomorphism of the annulus $N$, should not be confused with the standard definition of the rotation number for homeomorphisms of two dimensional domains (including the annulus) which very much depends on a particular homeomorphism, and not only on its homotopy class.
\end{remark}

Let $S$ denote  a compact Riemann surface (either closed or with  boundary).  Let $A \subset S$ be a topological annulus. Then $A$ has two ends. Moreover $A$ has two frontier components $\ph_0(A)$ and $\ph_1(A)$, each corresponding to one of the ends. Although the boundary $\ph{A}$ of $A$ is the union of $\ph_0(A)$ and $\ph_1(A)$, we do not call $\ph_0(A)$ and $\ph_1(A)$ the boundary components of $A$ because in general the sets $\ph_0(A)$ and $\ph_1(A)$ may not be disjoint. We call them frontier components of $A$.

\begin{definition}\label{twist-number-topological} Let $A \subset S$ be a topological annulus. Let $\wt{f}:\overline {S} \to \overline{S}$ be a homeomorphism such that $\wt{f}(A)=A$,
and such that $\wt{f}$ setwise fixes each of the two frontier components of $A$. Let $\Phi:N \to A$ be a conformal map and set $\wt{g}= \Phi^{-1} \circ \wt{f} \circ \Phi$.
\begin{itemize}
\item We define the twist number $\rho(\wt{f},A)$ to be equal to $\rho(\wt{g},N)$.
\item  We say that $\wt{f}$ has a conformal fixed point on  $\ph_i(A)$, $i=0,1$, if $\wt{g}$ has a fixed point on $\ph_i(N)$.
\end{itemize}

\end{definition}

Since $\wt{f}$ is a homeomorphism of $\overline{A}$ it follows that $\wt{g}$ is a homeomorphism of $\overline{N}$, so the twist number $\rho(\wt{g},N)$
is well defined. If $\wt{f}$ has conformal fixed points on both of its frontier components then $\rho(\wt{f},A)$ is an integer.

\begin{proposition}\label{uniform-twist}  Let $\wh{f}:\overline{P} \to \overline{P}$, be a homemorphism that setwise fixes the boundary lines of $P$, and that commutes with the translation $T(z)=z+1$. Then for every $z_0 \in \ph_0(P)$,  $z_1 \in \ph_1(P)$, we have
$$
\left| \big(Re( \wh{f}^n(z_1))- Re(\wh{f}^n(z_0) )\big) - (Re(z_1)- Re(z_0))  -n\rho(\wh{f},P) \right| < 6,
$$
\noindent
for every $n \in \N$.
\end{proposition}
\begin{proof} Let $\rho_1$ denote the translation number of the homeomorphism of the real line that is the restriction of $\wh{f}$ to $\ph_1(P)$. Similarly  $\rho_0$ denotes the translation number of the homeomorphism of the real line that is the restriction of $\wh{f}$ to $\ph_0(P)$.
Then Proposition 3.1 yields
$$
\left| \big(Re( \wh{f}^n(z_1))- Re(z_1) \big)  -n\rho_1 \right| < 3,
$$
\noindent
and
$$
\left| Re( \wh{f}^n(z_0))- Re(z_0)  -n\rho_0 \right| < 3.
$$
\noindent
Since  $\rho(\wh{f},P)=\rho_1-\rho_0$,  by subtracting the second inequality from the first we obtain
$$
\left| \big(Re( \wh{f}^n(z_1))- Re(\wh{f}^n(z_0) )\big) - (Re(z_1)- Re(z_0))  -n\rho(\wh{f},P) \right| <6.
$$

\end{proof}

For two smooth oriented arcs $h_1$ and $h_2$ such that the set $h_1 \cap h_2$ has finitely many points, by $\iota(h_1,h_2)$ we denote their algebraic intersection number. Let $l$ be an oriented Jordan arc in $N$ that connects the two boundary circles  $\ph_0(N)$ and $\ph_1(N)$. It is understood that such $l$ has one endpoint on each boundary circle, and the relative interior of the arc $l$ is contained in $N$. The homotopy class (modulo the endpoints) of $l$ is the collection of all such arcs that have the same endpoints as $l$ and are homotopic to $l$ in $N$, modulo the endpoints, and that are endowed with the orientation such that the endpoints have the same order with respect to this orientation. This class of arcs is denoted by $[l]$. Given two such arcs $l_1$ and $l_2$ (that may not have the same endpoints), by  $\iota([l_1],[l_2]) \in \Z$ we denote the algebraic intersection number between the two homotopy classes. This is well defined, since we can find  smooth representatives $h_j \in [l_j]$, $j=1,2$, such that the set $h_1 \cap h_2$ has finitely manus points,  and $\iota([l_1],[l_2])$ is defined as the algebraic intersection number $\iota(h_1,h_2)$ between these smooth arcs that is $\iota([l_1],[l_2])=\iota(h_1,h_2)$.
\vskip .1cm
Similarly let $A \subset S$ be a topological annulus. Let $l$ be an oriented Jordan arc that has one endpoint in each $\ph_0(A)$ and $\ph_1(A)$ (in particular this implies that the endpoints of $l$ are accessible points in the boundary of $A$). The homotopy class $[l]$ (modulo the endpoints) of $l$ is the collection of all such arcs homotopic to $l$ in $A$, that have the same endpoints as $l$, and with the corresponding orientations.
For two such arcs $l_1$ and $l_2$, by  $\iota([l_1],[l_2]) \in \Z$ we denote the algebraic intersection number between the two homotopy classes. This is well defined for the same reasons as above. Endow $M$ with a complex structure and let  $\Phi:N \to A$ be a surjective conformal map. Then $\Phi^{-1}(l)$ is a Jordan arc that connects the two boundary circles of $N$. We have $\iota([l_1],[l_2])=\iota([\Phi^{-1}(l_1)],[\Phi^{-1}(l_2)])$.
\vskip .1cm

The following proposition is elementary. The proof is left to the reader (see Figure \ref{stripe1}).

\begin{proposition}\label{intersection-explained} Let $A \subset N$ be a topological annulus homotopic to $N$ (we allow that $A=N$). Let $l_1,l_2 \subset A$ be two oriented Jordan arcs, each of them having one of its endpoints in each $\ph_0(A)$ and $\ph_1(A)$. Let $z_1,w_1 \in \overline{P}$ be the endpoints of a lift of $l_1$ to $P$ and let  $z_2,w_2 \in \overline{P}$ be the endpoints of a lift of $l_2$ to $P$. Let $h_1 \subset P$ be any smooth oriented arc with the endpoints $z_1,w_1$ and let $h_2 \subset P$ be any smooth oriented arc with the endpoints $z_2,w_2$. Then

$$
\iota([l_1],[l_2])=\sum_{k \in \Z} \iota(T^{k}(h_1),h_2).
$$

\end{proposition}

\begin{figure}
  \centering
  \includegraphics{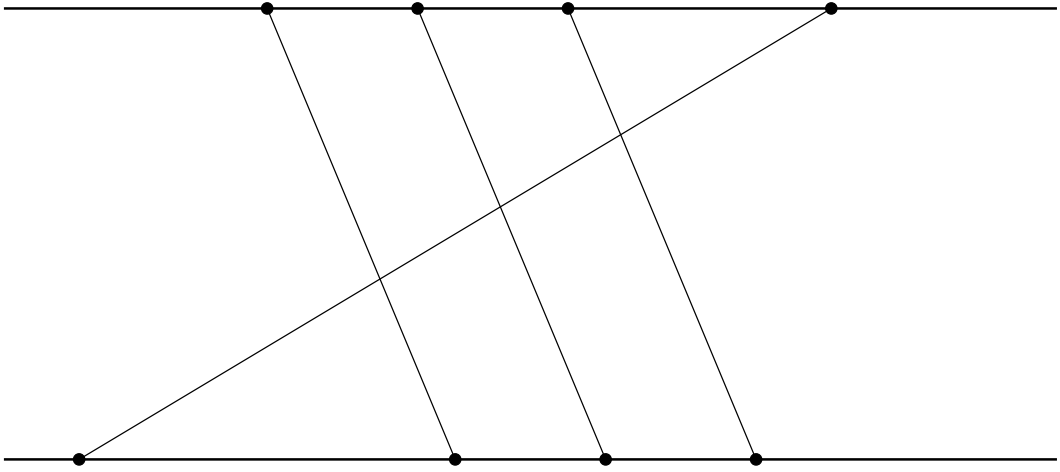}
  \caption{In this case the arcs $h_1$ and $h_2$ are straight lines in $P$.}\label{stripe1}
\end{figure}

\begin{proposition}\label{intersection-number}  Let $\wt{f}:\overline{N} \to \overline{N}$ be a homeomorphism that fixes setwise the boundary circles of $N$.  Let $l$ be any
oriented Jordan arc in $N$ connecting two boundary circles of $N$. Then
\begin{equation}\label{algebraic-intersection}
\left|\iota([l],[\wt{f}^{n}(l)]) -n\rho(\wt{f},N) \right| < 8,
\end{equation}
\noindent
for every $n \in \N$. In particular
$$
\rho (\wt{f},N)=\lim_{n\to\infty} {{\iota([l],[\wt{f}^{n}(l)]) }\over{n}}.
$$
\end{proposition}

\begin{remark} The same proposition holds for a homeomorphism $\wt{f}:\overline{A} \to \overline{A}$, where $A \subset S$ is a topological annulus.
\end{remark}
\begin{proof}
Let $\wh{f}$ be a lift of $\wt{f}$ to $P$. Let $\wh{l}$ be a single lift of $l$ to $P$. Let
$z_0 \in \partial_0 (P)$ and $z_1 \in \partial_1(P)$, be the endpoints of  $\wh{l}$. We compute the algebraic intersection number between the classes
$[l]$ and $[\wt{f}^{n}(l)]$ as follows. Replace the arcs $\wh{l}$ and $\wh{f}^{n}(\wh{l})$ by the straight lines that have the same endpoints as these two arcs, and denote these straight arcs by $h$ and $h_n$ respectiveley. Then $\iota([l],[\wt{f}^{n}(l)])$ is equal to the signed number of different translates of $h$ that intersect $h_n$ in $P$ (see the previous proposition and Figure \ref{stripe1}). The sign is equal to the intersection number between a single translate of $h$ (that intersects $h_n$)  and $h_n$. Then for every $n \in \N$ we have
$$
\left| \big(Re(\wh{f}^{n} (z_1))-Re(\wh{f}^{n}(z_0) ) \big)- (Re(z_1)-Re(z_0)) -  \iota([l],[\wt{f}^{n}(l)]) \right| \le 2.
$$
\noindent
Combining this inequality with Proposition \ref{uniform-twist} we obtain (\ref{algebraic-intersection}).
Divide this inequality by $n$, and  let $n \to \infty$. This proves  the rest of the proposition.
\end{proof}

\subsection{Long range Lipschitz maps on the strip} We have the following definition.
\begin{definition} Let $(X,\dis)$ be a metric space and let $x_1,x_2 \in X$. Let $F$ be a group of homeomorphisms of $X$.
We say that the group  $F$ is  $K$ long range Lipschitz  on the pair of points $x_1,x_2$, if $\dis(f(x_1),f(x_2)) \le K$
for every homeomorphism $f \in F$.
\end{definition}

\begin{remark} The constant $K$ in the above definition may depend on the choice of $x_1,x_2 \in X$.
\end{remark}

\begin{lemma}\label{long-range}  Let $\wt{f}:\overline{N} \to \overline{N}$ be a homeomorphism that setwise preserves the sets
$\ph_0(N)$ and $\ph_1(N)$, and such that $\wt{f} \circ \wt{e}=\wt{e} \circ \wt{f}$. Also, let
$C \subset N$ be a relatively closed set such that $\wt{f}(C)=\wt{e}(C)=C$, and every connected component of $C$  is compactly contained in $N$. Let $\wh{f}:\overline{P} \to \overline{P}$ be a lift of $\wt{f}$ and assume that   for every $z_1,z_2 \in (P \setminus \wh{C})$  the cyclic group generated by $\wh{f}$ is $K$ long range Lipschitz  on  the  pair $z_1,z_2$, for some constant $K=K(z_1,z_2)>0$ that depends on $z_1,z_2$ (here $\wh{C}$ is the lift of $C$ to $P$). Then $\rho(\wt{f},N)=0$.
\end{lemma}

\begin{proof}  There exists an integer $m \in \Z$, so that $\wh{f} \circ \wh{e}=T^{m} \circ \wh{e} \circ \wh{f}$. Note that $\wh{f}(\wh{C})=\wh{C}=\wh{e}(\wh{C})$. Since $\wh{f}$ commutes with the translation $T(z)$ we conclude that each $\wh{f}^{n}$, $n \in \Z$, is uniformly continuous on $\overline{P}$. That is, for a fixed $n \in \Z$ and  for every $\epsilon>0$, there exists $\delta=\delta(\epsilon,n)>0$ so that for every two points  $z_1,z_2 \in \overline{P}$ such that $|z_1-z_2| \le \delta$, we have $|\wh{f}^{n}(z_1)-\wh{f}^{n}(z_2)| \le \epsilon$.
\vskip .1cm
Let
$$
H(r)=\bigcup_{k \in \Z} \wh{f}^k(P(r) \setminus \wh{C}).
$$
\noindent
Since $\wh{f}$ commutes with $T(z)$ we have $T(H(r))=H(r)$, where $T(z)$ is the translation.

\begin{proposition} With the notation stated above we have the following. Assume that for some $1<r_1<r_0$ we have that
$H(r_1)$ has an accumulation point on $\ph{P}$. Then  $\rho(\wt{f},N)=0$.
\end{proposition}

\begin{proof}
Since $\wh{f}$ commutes with $T(z)$ there exists $w_1 \in (P(r_1) \setminus \wh{C})$ such that  $0 \le Re(w_1)<1$, and such that
$\wh{f}^{n_k}(w_1)$ converges  to $\partial P$, where $n_k$ is a sequence of integers.
Without loss of generality we may assume that
$$
\lim_{n_{k} \to \infty} Im(\wh{f}^{n_k}(w_1))={{\log r_0}\over{2\pi}},
$$
\noindent
that is the sequence $\wh{f}^{n_k}(w_1)$ converges  to $\ph_1(P)$.
Let $w_0=\wh{e}(w_1)$ (note that $w_0$ does not belong to $\wh{C}$).
Then $\wh{f}^{n_k}(w_0)$ converges  to $\ph_0(P)$ because $\wh{e}(\ph_1(P))=\ph_0(P)$.
Since  $w_1,w_0 \in (P \setminus \wh{C})$, we have  that the group generated by $\wh{f}$ is  $K$  long range Lipschitz
on the pair  $w_0,w_1$, for some constant $K>0$. This implies that for every $k \in \Z$ the Euclidean distance between the points $\wh{f}^{k}(w_1)$ and $\wh{f}^{k}(w_0)$
is bounded above by the constant $K$, that is
\begin{equation}\label{M-distance}
|\wh{f}^{k}(w_1)-\wh{f}^{k}(w_0)| \le K, \, k \in \Z.
\end{equation}
\vskip .1cm
Assume now that  $\rho(\wt{f},N) \ne 0$. Let $m_0 \in \Z$ be such that
\begin{equation}\label{m-0}
m_0\rho(\wt{f},N)>2K+11.
\end{equation}
\noindent
Since $\wh{f}^{m_{0}}$ is uniformly continuous on $\overline{P}$, there exists  $\delta'>0$
such that for every two points  $z,z' \in \overline{P}$, with $|z-z'| \le \delta'$, we have
\begin{equation}\label{m-0-1}
|\wh{f}^{m_{0}}(z)-\wh{f}^{m_{0}}(z')| \le {{1}\over{2}}.
\end{equation}
\noindent
Let  $\delta=\min\{{{1}\over{3}},\delta' \}$.
\vskip .1cm
Let $n_k$ be large enough so that the point $\wh{f}^{n_k}(w_1)$ is within the Euclidean distance
$\delta$ from $\ph_1(P)$. Then
$\wh{f}^{n_k}(w_0)$ is within the Euclidean distance $\delta$ from $\ph_0(P)$ (since $\wh{e}$ is an isometry).
Let $z_1 \in \ph_1(P)$ be a point such that $|\wh{f}^{n_k}(w_1)-z_1| \le \delta$, and  $z_0 \in \ph_0(P)$ be a
point such that $|\wh{f}^{n_k}(w_0)-z_0| \le \delta$.
We have
$$
|Re(z_1)-Re(z_0)| \le   |Re(\wh{f}^{n_{k}}(w_1) )-Re(\wh{f}^{n_{k}}(w_0))|+|\wh{f}^{n_{k}}(w_1)-z_1|+
|\wh{f}^{n_{k}}(w_0)-z_0| \le
$$
$$
\le K+2\delta<K+2.
$$
\noindent
Then it follows from Proposition 3.2 that
$$
\left| Re( \wh{f}^{m_{0}}(z_1))- Re(\wh{f}^{m_{0}}(z_0) ) -m_0\rho(\wh{f},P) \right| \le (K+2)+6<K+9,
$$
\noindent
which together with (\ref{m-0}) implies that
$$
|Re( \wh{f}^{m_{0}}(z_1))-Re (\wh{f}^{m_{0}}(z_0) ) | \ge m_0\rho(\wh{f},P)-K-9>K+2.
$$
\noindent
Again from the triangle inequality and from (\ref{m-0-1}) we get
$$
|Re(\wh{f}^{(n_{k}+m_{0})} (w_1))-Re(\wh{f}^{n_{k}+m_{0}}(w_0))| \ge  |Re( \wh{f}^{m_{0}}(z_1))-Re (\wh{f}^{m_{0}}(z_0) ) |-
$$
$$
-|Re(\wh{f}^{(n_{k}+m_{0})} (w_1))-\wh{f}^{m_{0}}(z_1)|- |Re(\wh{f}^{(n_{k}+m_{0})} (w_0))-\wh{f}^{m_{0}}(z_0)| \ge K+2-2{{1}\over{2}}=K+1.
$$
\noindent
But this contradicts (\ref{M-distance}).
\end{proof}

It remains to consider the case when every $H_r$ is a subset of $P(r')$ for some $r'<r_0$ (here $r'$ depends on $H_r$). The proof is by contradiction. From now on we assume that $\rho(\wh{f},P) \ne 0$.  Let $m \in\Z$ so that $\rho(\wh{f}^{m},P)>10$. Then for any pair of points $z_1,z_2 \in (P \setminus \wh{C})$ the group  generated by $\wh{f}^{m}$ is also  $K$ long range Lipschitz  on the pair $z_1,z_2$. So we may assume that $\rho (\wh{f},P)>10$.
\vskip .1cm
Since $\wh{f}$ is uniformly continuous on $\overline{P}$, there exists  $\delta'>0$   such that for every two points  $z,z' \in \overline{P}(r_0)$, with $|z-z'| \le \delta'$, we have
$|\wh{f}(z)-\wh{f}(z')| \le {{1}\over{2}}$.  Let  $\delta=\min\{{{1}\over{3}},\delta' \}$. Let $r_1<r_0$ be close enough to $r_0$ so that for every $z \in P \setminus P(r_{1})$, we have that the distance between $z$ and $\ph{P}$ is less than $\delta$. Then the same is true for every point in $ P \setminus H_{r_{1}}$ since $P(r_1) \subset H_{r_{1}}$.
\vskip .1cm
It follows from our assumption on the set $C$ that no connected component of $\wh{C}$ in $P$ accumulates on the boundary of $P$.
Therefore we can find $w_1 \in P \setminus (H_{r_{1}} \cup \wh{C})$  such that $Re(w_1)=0$. Let $\wh{e}(w_1)=w_0$. Then $w_0 \in P \setminus H_{r_{1}}$  and $Re(w_0)=0$. Let $z_1={{\log r_{0}}\over{2\pi}}$ and
$z_0=-{{\log r_{0}}\over{2\pi}}$. Then the Euclidean distance between
$z_i$ and $w_i$ is less than $\delta$, for $i=0,1$.
\vskip.1cm
We now show that $|\wh{f}^{n}(z_i)-\wh{f}^{n}(w_i)|<2n$.
The statement is true for $n=1$ by the choice of $\delta$.
We assume that it is true for $n$ and prove it for $n+1$.
Let $\zeta^{n}_i \in \ph_i(P)$ be the point with the same $x$-coordinate as  $\wh{f}^{n}(w_i)$. Then $|\zeta^{n}_i-\wh{f}^{n}(w_i)|<\delta$ since $\wh{f}^{n}(w_i) \in P \setminus H_{r_{1}}$ .  This implies that
\begin{equation}\label{induction}
|\wh{f}(\zeta^{n}_i)-\wh{f}^{n+1}(w_i)|<{{1}\over{2}}.
\end{equation}
Since $|\wh{f}^{n}(z_i)-\wh{f}^{n}(w_i)|<2n$ we conclude that
$$
|\zeta^{n}_i-\wh{f}^{n}(z_i)|= |Re\left(\wh{f}^{n}(w_i)-\wh{f}^{n}(z_i)\right)|< 2n.
$$
\noindent
This implies that $|\wh{f}(\zeta^{n}_i)-\wh{f}^{n+1}(z_i)|<2n+1$. The triangle inequality and (\ref{induction}) yield

$$
|\wh{f}^{n+1}(z_i)-\wh{f}^{n+1}(w_i)|<  |\wh{f}^{n+1}(z_i)-\wh{f}(\zeta^{n}_i)|+  |\wh{f}(\zeta^{n}_i)-\wh{f}^{n+1}(w_i)| \le 2n+{{3}\over{2}}<2(n+1).
$$
This proves the induction statement.
\vskip .1cm
We have that  $|\wh{f}^{n}(w_1)-\wh{f}^{n}(w_0)|\le K$, for every $n \in \Z$, where $K$ is such that the group generated by $\wh{f}$ is $K$ long range  Lipschitz   on $w_0$ and $w_1$. By the above induction statement and from the triangle inequality we have
$|\wh{f}^{n}(z_1)-\wh{f}^{n}(z_0)|\le K+4n$. Since  $\rho(\wh{f},P)>10$ we obtain a contradiction from Proposition \ref{uniform-twist}.

\end{proof}

\subsection{Fixed points of a  strip homeomorphism}
Let $\Omega \subset P$ be a simply connected domain that is invariant under the translation $T(z)=z+1$ (this means that $\Omega$ is the lift of a topological annulus $A \subset N$ with respect to the covering map $P \to N$). Then there exists $1<r_1 \le r_0$ and a conformal map $\Phi:P(r_1) \to P$, such that $\Phi$ commutes with the translation $T(z)=z+1$ and $\Phi(P(r_1))=\Omega$.

\begin{definition}\label{def-faithful}  Let $A \subset N$ be a topological annulus that is homotopic to $N$.
We say that $A$ is a faithful domain if
\begin{itemize}
\item We have $\inte(\overline{A})=A$, that is the interior of the closure of $A$ coincides with $A$.
\item Let $E$ be a connected component of the set $N \setminus \overline{A}$. Then the frontier $\ph{E}$ of $E$ is not contained in $\overline{A}$.
\end{itemize}
\end{definition}

If $\Omega \subset P$ is a  a simply connected domain that is invariant under the translation $T(z)$ we say that $\Omega$ is a faithful domain if
the corresponding annulus $A \subset N$ is a faithful domain. Equivalently this means that the interior of the closure of $\Omega$ coincides with $\Omega$,
and that if $E$ is a connected component of the set $P \setminus \overline{\Omega}$ then the relative frontier $\ph{E}$ of $E$ is not contained in $\overline{\Omega}$ (by the term relative frontier of $E$ we mean the frontier points of $E$ except $\infty$).
In Figure \ref{stripe2} we have an example of a domain $\Omega \subset P$ that is a simply connected domain, invariant under the translation $T(z)$, but that is not faithful.

\begin{remark} Let $A \subset N$ be a topological annulus homotopic to $N$. Let $\Omega$ be the lift of $A$ to $P$. Then $A$ is faithful if and only if $\Omega$ is faithful.
\end{remark}

\begin{proposition}\label{faithful-distance} Suppose that  $\Omega \subset P$ is a faithful and  simply connected set that is invariant under the translation $T(z)$. Let  $\Phi:P(r_1) \to P$, be a conformal map such that $\Phi$ commutes with $T(z)$ and $\Phi(P(r_1))=\Omega$. Then there exists a constant $C>0$ such that $|\Phi(z)-z| \le C$, for every $z \in P(r_1)$.
\end{proposition}

\begin{remark} This proposition does not hold  if $\Omega$ is not faithful (see Figure \ref{stripe2}).
\end{remark}

\begin{proof}  Let $\FD \subset P(r_1)$ be a rectangle fundamental
domain for the action of $T(z)$ on $P(r_1)$. If we prove that $\Phi(\FD)$  has a finite Euclidean diameter in $P$ then the proposition follows since $\Phi$ commutes with the translation $T(z)$.
\vskip .1cm
The proof is by contradiction. Assume that $\Phi(\FD)$  has an
infinite diameter. Then there exists a sequence $z_n \in \FD$ such
that the Euclidean distance between $\Phi(z_1)$ and $\Phi(z_n)$ goes to $\infty$ as $n\to\infty$. Set $w_n=\Phi(z_n)$. This implies that
$Re(w_n) \to \infty$. Without loss of generality we may assume that  $Re(w_{n}) \to +\infty$.
\vskip .1 cm
Let $l_{\infty}$ be the hyperbolic geodesic ray in $P(r_1)$ that  starts at $z_1$ and ends at $+\infty$. Let
$l_n$ be the hyperbolic geodesic ray that starts at $z_1$, and which contains $z_n$. After passing onto a subsequence if necessary, we have that $z_n$ converges to some point in $\overline{\FD}$.  Let  $l_{*}$ be the limit of the geodesic rays $l_n$. The geodesic ray $l_{*}$ starts at $z_1$
and  by $z_{*} \in \ph{P(r_1) }$ we denote the endpoint of $l_{*}$ on
$\ph{P(r_1)}$ (without loss of generality we may assume that $z_{*} \in \ph_1(P(r_1))$).
Note that $l_{*} \ne l_{\infty}$. Then $l_{\infty} \cup l_{*}$ divides $P(r_1)$ into two
simply connected sets $D_1$ and $D_2$.
We have that $z_{*}$ divides $\ph_1(P(r_1))$ into two Euclidean rays $\gamma_1$ and $\gamma_2$. The set  $D_1$ contains  $\gamma_1$ in its boundary and $D_2$ contains $\gamma_2$ in its boundary. Moreover
the boundary of $D_1$ is $l_{\infty} \cup l_{*} \cup \gamma_1 \cup \{+\infty\}$.
\vskip .1 cm
Let $l'_{\infty}=\Phi(l_{\infty})$ and
$l'_{*}=\Phi(l_{*})$. Since $\Phi$ commutes with the translation we see that $\Phi(+\infty)=+\infty$. We conclude that  $l'_{\infty}$
has $+\infty$ as its endpoint. It follows from  the assumption $Re(w_{n}) \to +\infty$ that $l'_{*}$ has $+\infty$ as its endpoint as well. However the curves $l'_{\infty}$ and $l'_{*}$ do not coincide because the curves $l_{\infty}$ and $l_{*}$ do not coincide either. Then the set $P \setminus (l'_{\infty} \cup l'_{*})$ has two simply connected components
$E_1$ and $E_2$, and
$$
\ph{E}_1=l'_{\infty} \cup l'_{*}.
$$
\noindent
Moreover $\Phi(D_1) \subset E_1$ and $\Phi(D_2) \subset E_2$.
\vskip .1cm
Let $\zeta_1 \in \gamma_1$, be an accessible point for $\Phi$, and let $\alpha_1$ be the hyperbolic geodesic ray from $z_1$ to $\zeta_1$.
Then   $\Phi(\alpha_1)$ is a finite diameter arc in $\Omega$.
Moreover  $\Phi(\zeta_1) \in E_1$ because $\Phi(D_1) \subset E_1$ and  $\Phi(\zeta_1)$ does not belong to $l'_{\infty} \cup l'_{*}$.
This shows that $E_1$ is not a subset of $\Omega$. If the set $E_1 \setminus \overline{\Omega}$ is non-empty then every connected component $O$ of this set is also a connected component of the set $P \setminus \overline{\Omega}$. Moreover $\ph{O}$ is contained in $\overline{\Omega}$ which is impossible since $\Omega$ is faithful. Therefore we have that the set $E_1 \setminus \overline{\Omega}$ is empty, that is $E_1 \subset \overline{\Omega}$. But since $\Omega$ is faithful and since $E_1$ is an open set we conclude that $E_1 \subset \Omega$. This is a contradiction since  $\Phi(\zeta_1) \in E_1$ and since $\Phi(\zeta_1)$ does not belong to $\Omega$.  So the set $\Phi(\FD)$ has a finite Euclidean diameter.

\end{proof}

\begin{figure}
\centering
\includegraphics{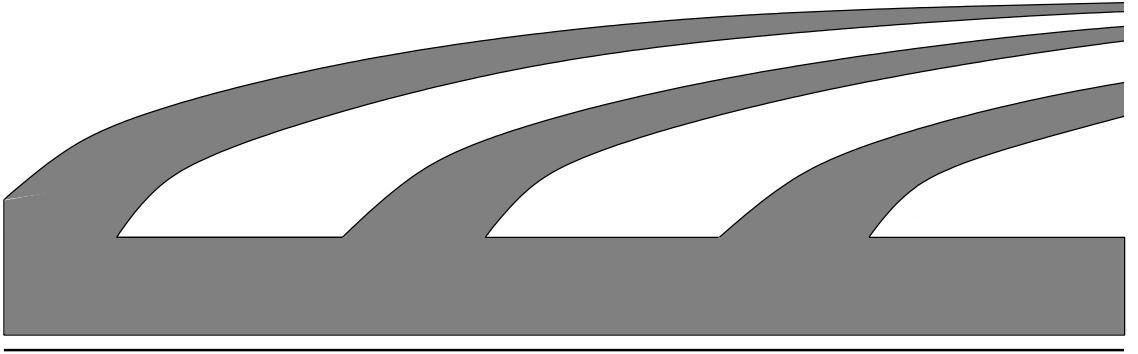}
\caption{This is an example of a simply connected domain that is
invariant under the translation $T(z)$ but that is not faithful.}
\label{stripe2}
\end{figure}

Fix a homeomorphism $\wh{f}:\overline{P} \to \overline{P}$ that commutes with the translation $T(z)$.
Let $1<r_1<r_0$ and let $\Phi:P(r_1) \to P=P(r_0)$ be a conformal map (it is conformal onto its image that is we do not assume that
$\Phi(P(r_1))=P$) that commutes with the translation $T(z)$,  that is $\Phi(z+1)=\Phi(z)+1$. Set $\Omega=\Phi(P(r_1))$ and assume that $\wh{f}$ setwise preserves $\Omega$.
Since $\wh{f}$ is a homeomorphism of $\overline{\Omega}$, we conclude  that the map $\wh{g}=\Phi^{-1} \circ \wh{f} \circ \Phi$ is a homeomorphism of $\overline{P(r_1)}$.  We define $\rho(\wh{f},\Omega)=\rho(\wh{g},P(r_1))$.
\vskip .1cm
First  we prove that under certain conditions on $\Omega$ and $\wh{f}$, the map $\wh{g}$ has a fixed point on $\ph{P(r_1)}$.

\begin{proposition}\label{fixed-point} Assuming the above notation we have the following. Suppose that  $\Omega$ is faithful.
If there exists a compact set $K \subset \ph{\Omega}$ which is setwised  fixed by $\wh{f}$, then the homeomorphism $\wh{g}$ has a fixed point on  $\ph{P(r_1)}$.
\end{proposition}

\begin{remark} It seems that this statement holds even if $\Omega$ is not faithful, but in that case the proof would be more involved.
\end{remark}

\begin{proof}  Let $\FD \subset P(r_1)$ be a rectangle fundamental
domain for the action of $T(z)$ on $P(r_1)$. In the previous proposition we have proved that $\Phi(\FD)$  has a finite Euclidean diameter in $P$.
We claim that the set
$$
\bigcup_{k\in\mathbb{Z}}T^k(\overline{\Phi(\FD)}),
$$
\noindent
covers $\overline{\Omega}$. Let $w_0 \in \overline{\Omega}$ and let $w_n \in \Omega$ be a sequence that converges to $w_0$. Then $w_n$ is a bounded sequence and since $\Phi(\FD)$ has a finite diameter we conclude that finitely many translates of $\Phi(\FD)$ contain all points $w_n$.
This implies that $w_0$ belongs to the closure of some translate of $\overline{\Phi(\FD)}$.
\vskip .1 cm
Since $K$ is a compact set, only finitely many translates of
$\overline{\Phi(\FD)}$ intersect  $K$.
Let $M_K=\{m \in \Z: \, T^m \big( \overline{\Phi(\FD)} \big) \cap K \ne \emptyset \}$. Then the homeomorphism $\wh{f}$ setwise fixes the set $\bigcup_{m \in M_K} T^m \big( \overline{\Phi(\FD)} \big)$. This implies that the homeomorphism $\wh{g}$ setwise fixes
$$
\ph_1(P(r_1)) \bigcap \left( \bigcup_{m\in M_K} T^m\big( \overline{\FD} \big)  \right).
$$
\noindent
Let $z_1,z_2 \in \ph_1(P(r_1))$ so that $Re(z_1)$ and $Re(z_2)$ are respectively  the infimum and the supremum of the set

$$
\ph_1(P(r_1)) \bigcap \left( \bigcup_{m\in M_K} T^m(\overline{\FD}) \right).
$$
\noindent
Then $-\infty<Re(z_1) \le Re(z_2)<+\infty$. We have that the homeomorphism $\wh{g}$ fixes $z_1$ and $z_2$. Therefore $\wh{g}$ has a fixed point.
\end{proof}

The following proposition is the version of Proposition \ref{uniform-twist} for faithful domains.

\begin{proposition}\label{uniform-twist-1} Let $1<r_1 \le r_0$ and let  $\Phi:P(r_1) \to P$ be a conformal map that commutes with the translation $T(z)$.
Set $\Phi(P(r_1))=\Omega$ and suppose that $\Omega$ is faithful domain. Let $\wh{f}:\overline{P} \to \overline{P}$ be a homemorphism that setwise fixes the boundary lines of $P$, that commutes with the translation $T(z)$, and such that $\wh{f}(\Omega)=\Omega$. Then there exists a constant $K>0$, such that for every $w_0 \in \ph_0(\Omega)$,  $w_1 \in \ph_1(\Omega)$, we have
$$
\left| \big(Re( \wh{f}^n(w_1))- Re(\wh{f}^n(w_0) )\big) - (Re(w_1)- Re(w_0))  -n\rho(\wh{f},\Omega) \right| < K.
$$
\noindent
for every $n \in \N$.
\end{proposition}
\begin{proof} Recall that $\wh{g}=\Phi^{-1} \circ \wh{f} \circ \Phi$. Let $C$ be the constant from Proposition \ref{faithful-distance}. We show that the proposition holds for $K=4C+6$.
\vskip .1cm
Assume first that $w_0$ and $w_1$ are accessible points. Let $\Phi^{-1}(w_i)=z_i \in \ph_i(P(r_1))$. Then by Proposition \ref{uniform-twist}  we have
$$
\left| \big(Re( \wh{g}^n(z_1))- Re(\wh{g}^n(z_0) )\big) - (Re(z_1)- Re(z_0))  -n\rho(\wh{g},P(r_1)) \right| < 6.
$$
\noindent
Since
$$
|w_i-z_i|, \, \, |\wh{f}^n(w_i))-\wh{g}^n(z_i))|  \le C, \quad i=0,1
$$
\noindent
we conclude that
$$
\left| \big(Re( \wh{f}^n(w_1))- Re(\wh{f}^n(w_0) )\big) - (Re(w_1)- Re(w_0))  -n\rho(\wh{f},\Omega) \right| < 4C+6=K.
$$
\noindent
The proposition now follows from the fact that the accessible points are dense in $\ph\Omega$.

\end{proof}

We have

\begin{lemma}\label{sum-twist-number} Let $\wt{f}:\overline{N} \to \overline{N}$ be a homeomorphism that setwise preserves the boundary circles of $N$. Let $A_i\subset N$, $i=0,1$, be two mutually disjoint toplogical annuli that are homotopic to $N$ such that $\ph_0(A_0)=\ph_0(N)$ and $\ph_1(A_1)=\ph_1(N)$. Suppose that the domains $A_i$, $i=0,1$ are  faithful domains. Let $Q=N \setminus (A_0 \cup A_1)$  and suppose $\wt{f}(A_i)=A_i$. If $\inte(Q)$ (the interior of $Q$) does not contain a component that is a topological annulus homotopic to $N$ then
$$
\rho(\wt{f},A_0)+\rho(\wt{f},A_1)=\rho(\wt{f},N).
$$
\noindent
If $A$ is a component of $\inte(Q)$ that is a topological annulus homotopic to $N$, and if $A$ is a faithful domain, then
$$
\rho(\wt{f},A_0)+\rho(\wt{f},A_1)+\rho(\wt{f},A) =\rho(\wt{f},N).
$$
\end{lemma}
\begin{proof} By $\wh{f}:\overline{P} \to \overline{P}$ we denote a lift of $\wt{f}$. Assume first that $\inte(Q)$ does not contain a component that is a toplogical annulus homotopic to $N$.  Since  $\inte(Q)$ does not contain a component that is a toplogical annulus homotopic to $N$ we conclude that there exists a point $r \in Q$ such that $r \in \ph{A_0} \cap \ph{A_1}$.  This point $r$ is fixed from now on.
Let $s_i \in \ph_i(N)$ be any two  points. Let $\wh{s}_i$, $i=0,1$, and $\wh{r}$ be lifts to $\overline{P}$ of the points
$s_i$, $i=0,1$, and $r$, respectively. Let $\Omega_i$, $i=0,1$, be the lifts of $A_i$ to $P$. Then $\wh{r} \in \ph\Omega_0 \cap \ph\Omega_1$.
We apply Proposition \ref{uniform-twist-1} to the pair $\wh{s}_0$ and $\wh{r}$, and then to the pair $\wh{s}_1$ and $\wh{r}$. This proves the proposition in this case.

\vskip .1cm
The case when $Q$ contains a component that is a toplogical annulus homotopic to $N$ is handled in a similar way. Assume that $\inte(Q)$ has a connected component $A$ that is a topological annulus homotopic to $N$. Since $Q$ is connected such $A$ is unique. Moreover there exist points $r_i \in \ph{A_i} \cap \ph{A}$, $i=0,1$.
Let $s_i \in \ph_i(N)$ be any two the points. Since the lifts of all three annuli are faithful domains the proof follows in the same way as above.

\end{proof}

\subsection{The $K$-idle set of an annulus homeomorphism}
We have the following definition.

\begin{definition} Let $\wt{f}:\overline{N} \to \overline{N}$ be a homeomorphism that setwise preserves the sets
$\ph_0(N)$ and $\ph_1(N)$. Let $w \in N$ and let  $K>0$.
We say that the point $w$ is $K$-idle for $\wt{f}$ if there exists a lift $\wh{f}:\overline{P} \to \overline{P}$  of $\wt{f}$ such that
\begin{equation}\label{K-idle}
|Re(\wh{f}^{n}(z))-Re(\wh{f}^{m}(z))| \le K,
\end{equation}
\noindent
for every $m, \, n \in \Z$, and for every $z \in P$ that is a lift of $w$.
The set of all $K$-idle points for $\wt{f}$ is denoted by $\Idle (K,\wt{f},N)$.
\end{definition}
Since $\wh{f}$ commutes with the translation $T(z)$ we have that  the condition (\ref{K-idle}) holds for one lift $z \in P$ of $w$ if and only if it holds for every such lift $z \in P$ of $w$. If $w \in \Idle (K,\wt{f},N)$ then there exists a unique lift $\wh{f}:\overline{P} \to \overline{P}$ such that
(\ref{K-idle}) holds. Also it follows from the definition that $\wt{f}(\Idle (K,\wt{f},N))=\Idle (K,\wt{f},N)$ (this follows from (\ref{K-idle}) ).
\vskip .1cm
\begin{remark} Let $\wt{e}$ be a conformal involution of $N$ and assume that $\wt{f}$ commutes with $\wt{e}$. Since every lift $\wh{e}:P \to P$ of $\wt{e}$ is an Euclidean isometry, we conclude that $\wt{e}(\Idle (K,\wt{f},N))=\Idle (K,\wt{f},N)$.
\end{remark}
\vskip .1cm
Let $w_i \in \Idle(K,\wt{f},N)$, $i=1,2$,  and let $\wh{f}_i$ denote the corresponding lift  so that the condition (\ref{K-idle}) holds for every lift of $w_i$ to $P$. We say that $w_1$ and $w_2$ are equivalent, $w_1 \sim w_2$, if $\wh{f}_1=\wh{f}_2$.

\begin{proposition}\label{idle-1} The set $\Idle (K,\wt{f},N)$  is a relatively closed subset of $N$. Let $Q$ be a connected component of $\Idle (K,\wt{f},N)$. Then every two points in $Q$ are equivalent. If in addition we have $\rho(\wt{f},N) \ne 0$ then $\overline{Q}$ does not connect the two boundary circles
$\ph_0(N)$ and $\ph_1(N)$.
\end{proposition}

\begin{proof}  Let $\wh{f}:\overline{P} \to \overline{P}$ be a lift of $\wt{f}$ and fix $n \in \Z$.  Since $\wh{f}$ commutes with the translation $T(z)$ we conclude that  $\wh{f}^{n}$ is uniformly continuous on $\overline{P}$. That is for every $\epsilon>0$  there exists $\delta=\delta(\epsilon,n)>0$ so that for every two points  $z_1,z_2 \in \overline{P}$ such that $|z_1-z_2| \le \delta$, we have  $|\wh{f}^{n}(z_1)-\wh{f}^{n}(z_2)| \le \epsilon$. Since every two lifts of $\wt{f}$ differ by a translation (which is an isometry in the Euclidean metric) we see that these constants do not depend on the choice of the lift $\wh{f}$.
\vskip .1cm
Fix $n_0 \in \N$ such that $n_0>K+3$ and let $\epsilon_0=1$. Set $\delta_0=\min\{\delta(\epsilon_0,n_0),1 \}$. Then for every lift $\wh{f}:\overline{P} \to \overline{P}$  we have
\begin{equation}\label{uniform-continuous}
|\wh{f}^{n_{0}}(z_1)-\wh{f}^{n_{0}}(z_2)| \le 1, \quad \quad \text{for every} \quad |z_1-z_2| \le \delta_0.
\end{equation}
\vskip .1cm
Let $w_i \in \Idle (K,\wt{f},N)$. Let $\wh{f}_i:\overline{P} \to \overline{P}$ be the lift of $\wt{f}$ such that the condition (\ref{K-idle}) holds with respect to $\wh{f}_i$. Then $w_1 \sim w_2$ if and only if $\wh{f}_1=\wh{f}_2$.
Assume that we can choose  lifts $z_1,z_2 \in P$, of $w_1$ and $w_2$ respectively, such that  $|z_1-z_2| \le \delta_0$. Then we show that $\wh{f}_1=\wh{f}_2$. This  implies that $w_1$ and $w_2$ are equivalent. In turn this  shows that the set $\Idle (K,\wt{f},N)$  is a relatively closed subset of $N$ and that every two points in $Q$ are equivalent.
\vskip .1cm
Next we show that $\wh{f}_1=\wh{f}_2$ providing that $|z_1-z_2| \le \delta_0$. Assume that $\wh{f}_1 \ne \wh{f}_2$. Then $\wh{f}_2 =T^{k} \circ \wh{f}_1$ for some $k \in \Z$, and $k \ne 0$. Since the condition (\ref{K-idle}) holds for $z_2$ and $\wh{f}_2$ we have
$$
|Re(\wh{f}^{n_{0}}_2(z_2))-Re(z_2)| \le K.
$$
\noindent
On the other hand we have
$$
|Re(\wh{f}^{n_{0}}_2(z_2))-Re(z_2)|=|kn_0+(Re(\wh{f}^{n_{0}}_1(z_2))-Re(z_2))| \ge |kn_0|- |Re(\wh{f}^{n_{0}}_1(z_2))-Re(z_2)|.
$$
\noindent
It follows from (\ref{uniform-continuous}) and the triangle inequality that (note that $\delta_0 \le 1$)
$$
|Re(\wh{f}^{n_{0}}_1(z_2))-Re(z_2)|\le
$$
$$
\le |Re(\wh{f}^{n_{0}}_1(z_2))-Re(\wh{f}^{n_{0}}_1(z_1))|+ |Re(\wh{f}^{n_{0}}_1(z_1))-Re(z_1)|+ |Re(z_1)-Re(z_2)|\le K+2,
$$
\noindent
that is
$$
|Re(\wh{f}^{n_{0}}_2(z_2))-Re(z_2)| \ge |kn_0|- (K+2)>|k|(K+3)-(K+2) \ge K+1.
$$
\noindent
But this is a contradiction so we conclude that $\wh{f}_1=\wh{f}_2$.
\vskip .1cm
Assume that $\rho(\wt{f},N) \ne 0$. Let $Q$ be a connected component of  $\Idle (K,\wt{f},N)$ and let $Q_1 \subset P$ be a connected component of the lift of $Q$ to $P$. Then by the previous argument we have that (\ref{K-idle}) holds for every $z \in Q_1$. By the continuity we have that
(\ref{K-idle}) holds for every $z \in \overline{Q_1}$. If $Q$ connects the two boundary circles of $N$ then $Q_1$ connects the two boundary lines of $P$.
Let $z_i \in \ph_i(P)$, $i=0,1$, be such that $z_i \in \overline{Q}_1$. Then for every $n \in \Z$ we have
$$
|Re(\wh{f}^{n}(z_1))-Re(z_1)| \le K, \quad i=0,1.
$$
\noindent
Since $\rho(\wt{f},N) \ne 0$ this contradicts  Proposition \ref{uniform-twist}.

\end{proof}

As above let $S$ denote  a compact Riemann surface (either closed or with  boundary).  Let $A \subset S$ be a topological annulus.
Let $\wt{f}:\overline {S} \to \overline{S}$ be a homeomorphism such that $\wt{f}(A)=A$. Let $\Phi:N \to A$ be a conformal map and set
$\wt{g}= \Phi^{-1} \circ \wt{f} \circ \Phi$. Since $\wt{f}$ is a homeomorphism of $\overline{A}$ it follows that $\wt{g}$ is a homeomorphism of $\overline{N}$. We define the corresponding set $\Idle(K,\wt{f},A) \subset A$ to be equal to $\Phi(\Idle(K,\wt{g},N))$.

\begin{proposition}\label{idle-2}  Let $\wt{f}:\overline{N} \to \overline{N}$ be a homeomorphism that has fixed points on both boundary circles of $N$ and that is homotopic to the standard twist modulo its fixed points on $\ph{N}$. Let $\wt{e}$ be a conformal involution of $N$ that exchanges the boundary circles of $N$ and suppose that $\wt{f}$ commutes with $\wt{e}$. Let  $A \subset N$ be a topological annulus homotopic to $N$, such that $\wt{e}(A)=\wt{f}(A)=A$, and assume that $\rho(\wt{f},A)$ is an odd integer. Fix $K>0$ and suppose that there exists a connected component $Q$ of the set  set  $\Idle(K,\wt{f},A)$ such that $Q$ is compactly contained in $A$ and that $Q$ separates the two frontier components of $A$.
Then there exists a topological annulus $A_1 \subset A$, and $A \ne A_1$, such that $\wt{e}(A_1)=\wt{f}(A_1)=A_1$, and $\rho(\wt{f},A_1)$ is an odd integer.
\end{proposition}

\begin{proof}
The set $Q$ is compact in $A$ and it separates the two frontier components of $A$ (and therefore $Q$ separates the two boundary circles of $N$).
Set $Q_1=\wt{e}(Q)$. Let $B'$ be union of  $Q \cup Q_1$ and every connected component of the set  $N \setminus (Q \cup Q_1)$ whose boundary is contained in  $Q \cup Q_1$.
Then $B'$ is a closed set that is compactly contained in $A$. Then the set $N \setminus B'$  is a disjoint union of two topological annuli $A'_0$ and $A'_1$ that  satisfy the conditions
\begin{itemize}
\item $\ph_0(A'_0)=\ph_0(N)$ and $\ph_1(A'_1)=\ph_1(N)$,
\item $\ph_1(A'_0), \ph_0(A'_1) \subset (Q \cup Q_1)$.
\end{itemize}
We construct the sets $A_i$, $i=0,1$, as follows. Let $A''_i$ be the union of $A'_i$ and every connected component of the set
$N \setminus \overline{A'}_i$ whose boundary is contained in $\overline{A'}_i$. Set $A_i=\inte(A''_i)$. Then
$A_1=\wt{e}(A_0)$. Note that $A_0 \cap A_1=\emptyset$ and each $A_i$ is a topological annulus  that is a faithful set in the sense of Definition \ref{def-faithful}. Set $B= N \setminus (A_0\cup A_1)$.  Note that $\ph_1(A_0), \, \ph_0(A_1) \subset (Q \cup Q_1)$.
\vskip .1cm
Recall that we are assuming that $\wt{f}$ has a fixed point on both boundary circles of $N$. This implies that $\wt{f}$ has a conformal fixed point on the corresponding frontier component of $A_i$ that agrees with the corresponding boundary circle of $N$. But $\wt{f}$ also has a conformal fixed point on the other frontier component of $A_i$, the one that is contained in $Q \cup Q_1$. This easily follows from Proposition \ref{fixed-point} and the assumption that $Q$ is a subset of $\Idle(K,\wt{f},A)$.
So we have that $\wt{f}$ has conformal fixed points on both frontier components of $A_i$, $i=0,1$.
Therefore we have that $\rho(\wt{f},A_i)$ is an integer. Since $A_1=\wt{e}(A_0)$ we have that  $\rho(\wt{f},A_0)=\rho(\wt{f},A_1)$ (because $\wt{f}$ commutes with $\wt{e}$). If $\inte(B)$ (the interior of $B$) does not contain an annulus homotopic to $N$ then
by Lemma \ref{sum-twist-number} we conclude that $\rho(\wt{f},N)= \rho(\wt{f},A_0)+\rho(\wt{f},A_1)=2\rho(\wt{f},A_0)$, that is $\rho(\wt{f},N)$ is an even integer which is a contradiction. So $\inte(B)$ contains an annulus homotopic to $N$. Denote this annulus by $D'$.
Let $D''$ be the union of $D'$ and every connected component of the set
$N \setminus \overline{D'}$ whose boundary is contained in $\overline{D'}$. Set $D=\inte(D'')$.
Then $D$ is a faithful domain and $\ph{D} \subset B'$. Moreover the annuli $A_0$, $A_1$ and $D$ are mutually disjoint.
Again by Lemma \ref{sum-twist-number} we have that $\rho(\wt{f},N)= \rho(\wt{f},A_0)+\rho(\wt{f},A_1)+\rho(\wt{f},D)=2\rho(\wt{f},A_0)+\rho(\wt{f},D)$. This implies that $\rho(\wt{f},D)$ is an odd integer. Since $D \subset A$ and $D \ne A$ we have found the required annulus which proves the proposition.

\end{proof}

\section{The minimal annulus}
\subsection{The groups $\Gamma(\ac_i,\bc_j)$ and the characteristic annulus}
From now on we assume that there exists a homomorphic section $\E:\MC(M) \to \Homeo(M)$, where $M$ is a surface of genus $\g \ge 2$.
By the end of the paper we will obtain a contradiction. From now on we fix the complex structure on $M$ so that the homeomorphism $\E(e)$ is a conformal involution (it is well known that such a structure exists).
\vskip .1cm
Recall from Section 2.1 the definition of curves $\alpha_i,\beta_j$, and $\gamma$ (and $\gamma_1$ in case when $\g$ is odd).
For a fixed pair $(i,j) \in \{(1,1),(1,2),(2,1),(2,2) \}$ let
$\Gamma(\ac_i,\bc_j)$ be the group generated by the homeomorphisms
$\E(t_{\alpha_k}), \, \E(t_{\beta_l})$, where $\alpha_k \in \ac_i$ and   $\beta_l \in \bc_j$.
Let $\Se(i,j)$ denote the minimal decomposition of $M$ for the group  $\Gamma(\ac_i,\bc_j)$, and let
$M_{\Se(i,j)}$ be the union  all acyclic components of  $\Se(i,j)$.
By $\Se(\gamma)$ we denote the minimal decomposition for the homeomorphism $\E(t_{\gamma})$, and by
$M_{\Se(\gamma)}$ we denote the union of all acyclic components in $\Se(\gamma)$.
\vskip .1cm
It follows from Lemma 2.1 that if $\gamma$ is a separating curve (that is $\g$ is even)
then there exist precisely two non-planar components of
$M_{\Se(\gamma)}$. Each component has exactly one end and this end is
homotopic to $\gamma$. Let $M_{\gamma}$ be the union of these two
components. If $\gamma$ is non-separating, then $M_{\Se(\gamma)}$ has
precisely one non-planar component, and this component has two ends that are both
homotopic to $\gamma$. Let $M_{\gamma}$ be this non-planar component in this case.
\vskip .1cm
Then $\E(t_{\gamma})$ setwise preserves $M_{\gamma}$. Since the involution $e$ commutes with $t_{\gamma}$ we have that
$\E(e)$ setwise preserves $M_{\gamma}$. If $\g$ is even then $\E(e)$ exchanges the two components of $M_{\gamma}$ and if $\g$ is odd then $\E(e)$ exchanges the two ends of $M_{\gamma}$. Set $B_{\gamma}=M \setminus M_{\gamma}$. In both cases ($\g$ odd or even)
we have that $B_{\gamma}$ is a closed, connected and  non-acyclic subset of $M$, and $\E(t_{\gamma})(B_{\gamma})=\E(e)(B_{\gamma})=B_{\gamma}$.
\vskip .1 cm
Let $\gamma'\subset M_{\gamma}$ and $\gamma''=\E
(e)(\gamma')$ be two simple closed curves on $M$
representing the two ends of $M_{\gamma}=M \setminus B_{\gamma }$. Then
$M_{\gamma} \setminus \{\gamma',\gamma''\}$ has either three or four
components depending whether $M_{\gamma}$ is connected or not. We
denote by $M_{\gamma}(\gamma',\gamma'')$ either the component, or the union
of two components of $M_{\gamma} \setminus \{\gamma',\gamma''\}$
whose boundary consists only of $\gamma'$ and $\gamma''$. Then
$$
\overline{M_{\gamma}(\gamma',\gamma'')}=
M_{\gamma}(\gamma',\gamma'')\cup \gamma' \cup \gamma''
$$
\noindent
is the closure
of $M_{\gamma}(\gamma',\gamma'')$. Note that $\E(e)$ setwise preserves $M_{\gamma}(\gamma',\gamma'')$ (but $\E(t_{\gamma})$ does not necessarily preserve this set).
\vskip .1cm
Let $C_{\gamma}(\gamma',\gamma'')$ denote the union of all acyclic
components of $\Se(\gamma)$ which intersect $\overline{M_{\gamma}(\gamma',\gamma'')}$.
The set  $C_{\gamma}(\gamma',\gamma'')$ is a compact subset of $M_{\gamma}$
which is either connected or it has two components depending whether
$\gamma$ is a non-separating or a separating curve. Note that $\E(t_{\gamma})$ setwise preserves the set
$C_{\gamma}(\gamma',\gamma'')$ (because this set is a union of components from $\Se(\gamma)$ and these by definition are setwise preserved by
$\E(t_{\gamma})$). Also, $\E(e)$ setwise preserves $C_{\gamma}(\gamma',\gamma'')$. This follows from the fact that $\E(e)$ setwise preserves $M_{\gamma}(\gamma',\gamma'')$, and since the homeomorphisms $\E(e)$ and $\E(t_{\gamma})$ commute, we have that if $S \in M_{\Se(\gamma)}$
then $\E(e)(S) \in  M_{\Se(\gamma)}$ as well.
\vskip .1cm
Let
$A_{\chara}$ denote the component of the complement of $C_{\gamma}(\gamma',\gamma'')$  that contains
$B_{\gamma}$. Then $A_{\chara}$ is an open topological annulus (or just annulus) homotopic to $\gamma$.
The annulus $A_{\chara}$ is called the characteristic annulus (see \cite{ma}).
We have
\begin{equation} \label{A-0}
\E(t_{\gamma})(A_{\chara})=\E(e)(A_{\chara})=A_{\chara}.
\end{equation}
\noindent
Let $\Phi:N \to A_{\chara}$ be a conformal map to the corresponding geometric annulus $N$.
By \cite[Lemma 5.1]{ma}, the map $\Phi^{-1} \circ \E (t_{\gamma}) \circ \Phi:\overline{N}\to \overline{N}$ has at least one fixed point on each boundary
component of $\overline{N}$ and it is homotopic (modulo its fixed points on the boundary) to the standard twist homeomorphism.
\vskip .1cm
The annulus $A_{\chara}$ depends on the choice of $\gamma'$ and $\gamma''$ but this is not relevant in this paper
(that is we will not be changing the choice of these two curves). From now on $A_{\chara}$ is fixed and we will return to it later.

\begin{lemma}\label{minimal-1}  Fix  $i,j \in \{1,2\}$.  The set   $M_{\Se(i,j)}$ contains a unique connected component  $M_{\Se(i,j)}(\gamma)$ with the following properties
\begin{itemize}
\item  $M_{\Se(i,j)}(\gamma)$ has negative Euler characteristic.
\item $M_{\Se(i,j)}(\gamma)$ contains a curve homotopic to $\gamma$.
\item No end of $M_{\Se(i,j)}(\gamma)$ is homotopic to $\gamma$ and every end of $M_{\Se(i,j)}(\gamma)$ is essential in $M$ (this means that
every end of  $M_{\Se(i,j)}(\gamma)$ is homotopic to a simple closed curve in $M$ that is not homotopically trivial).
\end{itemize}
\end{lemma}
\begin{remark} The homeomorphism $\E(e)$ permutes the sets $M_{\Se(i,j)}(\gamma)$ that is  $\E(e)(M_{\Se(i,j)}(\gamma))=M_{\Se(j,i)}(\gamma)$.
\end{remark}

\begin{proof}
Observe that we may choose the curves in $\ac_i$ and $\bc_j$ such that every two curves from $\ac_i \cup \bc_j$ are mutually disjoint  and any curve from $\ac_i \cup \bc_j$  is disjoint from the curve $\gamma$ (see Figure \ref{genus-2}, Figure \ref{genus-4}, Figure \ref{genus-3}). Then the set $M \setminus
(\ac_i \cup \bc_j)$ contains a component $C$ of negative Euler
characteristic such that $\gamma \subset C$ and that no end of $C$ is homotopic to $\gamma$. Moreover every end of $C$ is essential in $M$. By Lemma 2.1 we conclude that $M_{\Se(i,j)}$ has a component $M_{\Se(i,j)}(\gamma)$ homotopic to $C$.
\end{proof}

\begin{proposition}\label{minimal-2} Let $A \subset M$ be an annulus with the following properties
\begin{enumerate}

\item The annulus $A$ is homotopic to $\gamma$.
\item $\E(t_{\gamma})(A)=A$.
\item The twist number $\rho(\E(t_{\gamma}),A)$ is an odd integer.
\end{enumerate}
Let $\Gamma(\ac_i,\bc_j)$ be one of the four groups we defined above.
Then there exists a unique connected component $A_1$ of the set $A \cap M_{\Se(i,j)}(\gamma)$ such that that  $A_1$ is a topological annulus homotopic to $\gamma$. Moreover $A_1$  satisfies the properties $(1)$, $(2)$ and $(3)$. In fact we have  $\rho(\E(t_{\gamma}),A)=\rho(\E(t_{\gamma}),A_1)$.
\end{proposition}

\begin{proof} We first study the set $A \cap (M \setminus M_{\Se(i,j)}(\gamma))$.
Since $ M_{\Se(i,j)}(\gamma)$ contains a curve homotopic to $\gamma$ and $\gamma$ is not homotopic to any end of $ M_{\Se(i,j)}(\gamma)$,
we conclude that the set $M \setminus M_{\Se(i,j)}(\gamma)$ does not contain a curve homotopic
to $\gamma$.  For $\epsilon>0$ let $B_{\epsilon}$ denote an $\epsilon$-neighbourhood (with respect to the hyperbolic metric on $M$ that we fixed above)
of the set $M \setminus M_{\Se(i,j)}(\gamma)$. Then we may choose $\epsilon$ small enough so that the set $B_{\epsilon}$ does not contain a curve homotopic
to $\gamma$. Fix  such $\epsilon>0$. Then no connected component of the set $A \cap B_{\epsilon}$ can separate the two frontier components  of $A$.
Now we show that no connected component of  the set $A \cap B_{\epsilon}$ can connect the two frontier components of $A$.
\vskip .1 cm
Let $0<\delta < \epsilon$ be such that
$$
\dis((\E(t_{\gamma}))^{10}(x),(\E (t_{\gamma}))^{10}(y))<\epsilon,
$$
\noindent
whenever $x,y\in M$ and $\dis(x,y)<\delta$.
This implies $(\E(t_{\gamma}) )^{10}(B_{\delta}) \subset B_{\epsilon}$.
Assume that there exists a connected component of the set $A \cap B_{\delta}$ that connects the two frontier components of $A$.
Then there exists a Jordan arc $l \subset \overline{A} \cap B_{\delta}$ that connects the two boundary components $\ph_0(A)$ and $\ph_1(A)$ of $A$.
Since $\E (t_{\gamma})$ has an odd rotation number and it fixes
$A$, it follows from (\ref{algebraic-intersection})   that $|\iota([l],[\E(t_{\gamma})^{10}(l)])| \ge 2$.
This implies that $\big( (\E (t_{\gamma}))^{10}(l) \cup l \big)$ separates the two frontier components of $A$.
This is a contradiction since $\big( (\E (t_{\gamma}))^{10}(l) \cup l \big)$ is contained in $B_{\epsilon}$.
Therefore no connected component of the set $A \cap B_{\epsilon}$ can connect the two ends of $A$.
\vskip .1cm
We have that no component of the set $A \cap (M \setminus M_{\Se(i,j)}(\gamma))$ can separate or connect the two frontier components of $A$. Also the sets
$$
\ph_0(A) \setminus \big( M \setminus M_{\Se(i,j)}(\gamma) \big), \quad \text{and} \quad  \ph_1(A) \setminus \big( M \setminus M_{\Se(i,j)}(\gamma) \big),
$$
\noindent
are non-empty and relatively open. Combining this with the fact that $M_{\Se(i,j)}(\gamma)$ is connected we conclude that the set $A \cap M_{\Se(i,j)}(\gamma)$ contains a  component that is an annulus homotopic to $\gamma$.
Since $M_{\Se(i,j)}(\gamma)$ is connected such annulus is unique. Denote this annulus by $A_1$. We show that $A_1$ satisfies the same properties as $A$.
\vskip .1cm
We have already seen that $A_1$ is homotopic to $\gamma$.
Since $\E(t_{\gamma})$ setwise preserves $M_{\Se(i,j)}(\gamma)$ we conclude that $\E(t_{\gamma})$ setwise preserves $A_1$.
Since the sets $\ph_0(A) \setminus (M \setminus M_{\Se(i,j)}(\gamma))$ and $\ph_1(A) \setminus (M \setminus M_{\Se(i,j)}(\gamma))$
are non-empty there exists a Jordan arc $l \subset A_1$ that connects the two frontier components components of $A$. It follows from Proposition \ref{intersection-number}  that $\rho(\E(t_{\gamma}),A)=\rho(\E(t_{\gamma}),A_1)$.
\end{proof}

\subsection{The minimal annulus}
We have
\begin{definition} Let $A_{\chara}$ be the characteristic annulus defined above. Let $A$ be a topological annulus on $M$ homotopic to
$\gamma$. The annulus $A$ is said to be an $admissible$ $annulus$
if it satisfies the following
\begin{enumerate}
\item $A \subset A_{\chara}$.
\item $\E (e)(A)=A$
\item $\E (t_{\gamma})(A)=A$
\item The twist number $\rho (\E (t_{\gamma}),A)$ is an odd integer.
\item $A \subset M_{\Se(i,j)}(\gamma)$ for any pair $i,j \in \{1,2\}$.
\end{enumerate}
A topological annulus $A$ is said to be a $minimal$ $annulus$ if it is admissible and if no other
admissible annulus is strictly contained in $A$.
\end{definition}

\begin{proposition}\label{minimal-3} There exists an admissible annulus.
\end{proposition}

\begin{remark} In fact we have that  $\rho(\E(t_{\gamma}),A)=1$ for the admissible annulus we construct in the proof below..
\end{remark}

\begin{proof} Consider the set

$$
B=A_{\chara} \bigcap \left(\bigcap_{i,j \in \{1,2\}} M_{\Se(i,j)}(\gamma) \right).
$$
\noindent
By Proposition \ref{minimal-2} there exists a unique connected component $A(1,1)$ of $A_{\chara} \cap M_{\Se(1,1)}(\gamma)$ that is an annulus homotopic to $\gamma$. Moreover $\E(t_{\gamma})$ setwise preserves this annulus and $\rho(\E(t_{\gamma}),A(1,1))=\rho(\E(t_{\gamma}),A_{\chara})=1$. That is the annulus $A(1,1)$ satisfies the assumptions of Proposition 4.1. We apply this proposition again and find that
$A(1,1) \cap M_{\Se(1,2)}(\gamma)$ contains a  unique component $A(1,2)$ that is an annulus homotopic to $\gamma$. Again $A(1,2)$ satisfies the assumptions of Proposition \ref{minimal-2}. We repeat this two more times.
We conclude that the set $B$ has a component $A$ that is an annulus homotopic to $\gamma$, such that
$\E (t_{\gamma})(A)=A$, and such that $\rho (\E (t_{\gamma}),A)=1$.
\vskip .1cm
It remains to show that $\E (e)(A)=A$. Set $A_1=\E (e)(A)$. It follows from (\ref{A-0}) (also see Lemma \ref{minimal-1}) that $\E(e)$ setwise preserves the set $B$. We conclude that $A_1$ is also a connected component of the set $B$. In particular we have that either $A=A_1$ or
$A \cap A_1=\emptyset$, and  both annuli $A$ and $A_1$ are homotopic to $\gamma$. We show that $A=A_1$. If $A \cap A_1=\emptyset$ then the set $M \setminus (A \cup A_1)$ contains a unique component $B'$ such that $A \cup A_1 \cup B'=A'$ is an annulus homotopic to $\gamma$. Moreover, the boundary of $A'$ is a subset of $\ph{A} \cup \ph{A_1}$. Since each subsurface $M_{\Se(i,j)}(\gamma)$ contains $A \cup A_1$ and since the ends of  $M_{\Se(i,j)}(\gamma)$ are essential (see Lemma 4.1) we conclude that $A' \subset M_{\Se(i,j)}(\gamma)$ for each $i,j$. This shows that $A=A_1$.
\end{proof}

\begin{proposition}\label{minimal-4} There exists a minimal annulus.
\end{proposition}
\begin{remark} It can be shown that in fact there exists a unique minimal annulus. We do not need this result so we omit proving it.
\end{remark}
\begin{proof}
Consider the  family $\mathcal{F}$ of all admissible annuli. This family is non-empty by the previous proposition.
The partial ordering on $\mathcal{F}$ is given by the inclusion.
By the Zorn's lemma  there exists a maximal chain $\mathcal{A}$ in
$\mathcal{F}$. We show that the intersection of all annuli in
$\mathcal{A}$ is a set whose interior  contains an annulus that  also belongs to $\mathcal{A}$ (that is we show that
$\mathcal{A}$ has the minimal element). This annulus is then by definition a minimal annulus.
\vskip .1 cm
Since $M$ is a separable space, it follows that there exists a
decreasing sequence of annuli $A_n \in \mathcal{C}$ such that
$$
\bigcap_{n \in \N}A_n=\bigcap_{A \in \mathcal{A}}A.
$$
\noindent
Let $\dis$ denote the corresponding hyperbolic distance on $M$ (recall that we have fixed the complex structure on $M$).
Let $D \subset A_n$ be a geodesic disc (with respect to the hyperbolic metric). We say that $D$ is a proper maximal disc if the closed disc $\overline{D}$ has non-empty intersection with both frontier components of $A_n$.

\begin{remark} If $\Omega \subset M$ is a domain we say that $D \subset \Omega$ is a maximal disc in $\Omega$ if $D$ is a geodesic disc  that is not contained in any larger hyperbolic disc which is a subset of $\Omega$. Then the closed disc $\overline{D}$ has to touch the boundary of $\Omega$. If $\Omega$ is a topological annulus then a closed  maximal disc does not need to connect the two frontier components of $A_n$. This is why we call such discs proper maximal discs.
\end{remark}

If $z \in \ph{A}_n$ is a point where $\overline{D}$ touches the boundary $\ph{A}_n$ then $z$ is an accessible point, and the geodesic arc that connects the centre of $D$ with $z$ is contained in $A$. Let $D_n$ be a proper maximal disc that has the smallest radius among all proper maximal discs in $A_n$ (there could be more than one such disc with the smallest radius and we pick one). Let $r_n$ denote the radius of $D_n$ and let $c_n$ be its centre. We show that
$\liminf\limits_{n \to \infty} r_n>0$.
\vskip .1cm
Since $D_n$ is a proper maximal disc there exist points $z_n \in \ph_0(A_n)$ and $w_n \in \ph_1(A_n)$ that are in the closed disc $\overline{D}_n$.
Let $l_n \subset D_n \subset A_n$ be the arc that has the endpoints $z_n$ and $w_n$, such that $l_n$ is the union of the two geodesic arcs
that connect $z_n$ and $w_n$ with $c_n$ respectively.
Let $l'_n=(\E (t_{\gamma}))^{10} (l_n)$. Since $\rho (\E(t_{\gamma}), A_n)$ is odd, it follows from Proposition \ref{intersection-number} (and formula (\ref{algebraic-intersection}) ) that $|\iota([l_n],[l'_n])| \ge 2$. This implies that the set $l_n \cup l'_n$ contains a closed curve homotopic to $\gamma$. Denote this curve by $t_n$. For every $n \in \N$ the hyperbolic diameter of $t_n$
is bounded below by the half of the hyperbolic length of the simple closed geodesic homotopic to $\gamma$.
\vskip .1cm
Assume that after passing onto a subsequence if necessary, we have $r_n \to 0$, $n \to \infty$. Then the
hyperbolic diameter of the arc $l_n$ tends to zero. Since $\E(t_{\gamma})^{10}$ is uniformly continuous on $M$ we see that the hyperbolic diameter of $l'_n$ tends to zero as well. Since $t_n \subset (l_n \cup l'_n)$ we conclude that the hyperbolic diameter of $t_n$ tends to zero. But this is a contradiction with the fact that the the hyperbolic diameter of $t_n$ is bounded away from zero.
\vskip .1cm
Let $1<s_n$ so that
$$
\Mod(A_n)={{\log s_n}\over {2\pi}}.
$$
\noindent
Here $\Mod(A_n)$ denotes the conformal modulus of $A_n$. Then there exists a conformal map $\Phi_n: N(s_n) \to M$ such that $\Phi_n(N(s_n))=A_n$. Since the radius of every proper maximal disc is bounded away from zero (regardless of $n$) we conclude that the distance between the two frontier components of $A_n$ is bounded away from zero (regardless of $n$).
Combining this with the fact that the sequence $A_n$ is decreasing we have that $\lim\limits_{n \to \infty} s_n=s$ exists and $1<s$ (see \cite{l-v}).  This shows that the sequence of conformal maps $\Phi_n$ converges on every compact set in $N(s)$ (after passing onto a subsequence if necessary) to a non-degenerate conformal map $\Phi:N(s) \to M$ (note that every compact set in $N(s)$ is eventually contained in $N(s_n)$ which is the domain of $\Phi_n$). Let $A=\Phi(N(s))$. Then $A$ is a topological annulus that is contained in $\bigcap_{n \in \N}A_n$. Clearly $A$ is setwise preserved by $\E(t_{\gamma})$ and $\E(e)$. It remains to show that the twist number $\rho(\E(t_{\gamma}),A)$ is an odd integer.
\vskip .1cm
Going back to the sequence of proper maximal discs $D_n$ we see that after passing to a subsequence if necessary, we have that $D_n \to D$ where $D$ is a proper maximal discs in $A$. Also the sequence of arcs $l_n$ converges to the corresponding arc $l \subset A$ whose endpoints are in the opposite frontier components of $A$ (recall that each $l_n$ constitutes of the two geodesic arcs and so does $l$). Fix $k \in \N$. Then by the continuity of $(\E(t_{\gamma}))^{k}$, for $n$ large enough we have that
$$
|\iota([l_n],[(\E(t_{\gamma}))^{k}(l_n)])- \iota([l],[(\E(t_{\gamma}))^{k}(l)])| \le 2.
$$
\begin{remark} If $(\E(t_{\gamma}))^{k}$ does not fix either endpoint of $l$ then for $n$ large enough we have that
$\iota([l_n],[(\E(t_{\gamma}))^{k}(l_n)])= \iota([l],[(\E(t_{\gamma}))^{k}(l)])$. But if $(\E(t_{\gamma}))^{k}$ fixes one or both endpoints of $l$ then the two numbers may differ by $2$.
\end{remark}
By the previous inequality and from Proposition \ref{intersection-number} (formula (\ref{algebraic-intersection}) ) we have
$$
|k\rho((\E(t_{\gamma})),A_n)-k\rho((\E(t_{\gamma})),A)| \le 20,
$$
\noindent
for every $k \in \N$ and $n$ large enough. We have
$$
\lim\limits_{n \to \infty} \rho(\E(t_{\gamma} ),A_n)=\rho(\E(t_{\gamma}),A).
$$
\noindent
Since every $\rho(\E(t_{\gamma}),A_n)$ is an odd integer so is $\rho(\E(t_{\gamma}),A)$.

\end{proof}

\subsection{The action of the twist $\E(t_{\gamma})$ on the minimal annulus}
From now on we fix a minimal annulus and call it $A_{\min}$. We may assume that the strip $P$ is the universal cover of $A_{\min}$.
Let $p \in A_{\min}$ and let $S_p(i,j) \in \Se(i,j)$ be the corresponding component that contains $p$ (here $i,j \in \{1,2\}$). Let $\wh{S}_p(i,j)$ denote a single lift of $S_p(i,j)$  to $P$. Since  $S_p(i,j)$ is acyclic we have that every connected component of the set $P \cap \wh{S}_p(i,j)$ has a finite Euclidean diameter (if a relatively closed subset of $P$ has an infinite diameter, and if this set is invariant for the translation for $1$, then the projection of this set to $A_{\min}$ is not acyclic). Let $C$ be the supremum of such Euclidean diameters when $p \in A_{\min}$ and $i,j \in \{1,2\}$. Since  $\Se(i,j)$ is upper semi-continuous we find that this supremum is achieved and we denote it by $C_{\min}$.

\begin{proposition} \label{minimal-5} Let $S \in \Se(i,j)$, $i,j \in \{1,2\}$, and assume that $S$ has a non-empty intersection with  $A_{\min}$. Then $S$ is acyclic and $S$ is compactly contained in $A_{\min}$.
\end{proposition}

\begin{proof} It follows from the definition of  $A_{\min}$ that for each point $p \in A_{\min}$, the corresponding component $S_p(i,j) \in \Se(i,j)$ that contains $p$ is acyclic. Since $\E (t_{\gamma})$ commutes with the elements of $\Gamma(\ac_i,\bc_j)$, it follows that
$\E (t_{\gamma})$ permutes the components of the minimal decompositions $\Se(i,j)$, that is $\E (t_{\gamma})(S_p(i,j))=S_{\E (t_{\gamma})(p)}(i,j)$.
First we show that $S_p(i,j)$ cannot connect the two frontier components of $A_{\min}$.
\vskip .1cm
Assume on the contrary that $S_p(i,j)$ connects the two frontier  components of $A_{\min}$. Let $S$ be a connected component of $S_p(i,j) \cap A_{\min}$
such that $\overline{S}$ connects the two frontier components. Let $\wh{S}$ be a lift of $S$ to $P$. Then the closure $\overline{\wh{S}}$ of $\wh{S}$ connects the  two boundary lines of $P$. Let $z_i \in \ph_i(P) \cap \overline{\wh{S}}$, $i=0,1$. Let $\wh{f}$ be a lift of $\E(t_{\gamma})$ to $P$.
Then $|\wh{f}^{n}(z_1)-\wh{f}^{n}(z_2)| \le C_{\min}$, for every $n \in \Z$. By Proposition \ref{uniform-twist} we have that $\rho(\E(t_{\gamma}),A_{\min})=0$.
\vskip .1cm
Next we show that $S_p(i,j)$  cannot intersect the boundary of
$A_{\min}$.  Assume on the contrary that $S_p(i,j)$  intersects the frontier component $\ph_0(A_{\min})$.
We already showed that $S_p(i,j)$  cannot connect the two frontier components of $A_{\min}$.
Since $\E (t_{\gamma})$ preserves each frontier component of $A_{\min}$, it follows that the set
$$
X=\bigcup_{k\in\mathbb{Z}}\E (t_{\gamma})^k(S_p(i,j) \bigcap\overline{A_{\min}}),
$$
\noindent
intersects only  $\ph_0(A_{\min})$. Moreover the closure $\overline{X}$ cannot
connect the two frontier components of $A_{\min}$ either. If we assume that $\overline{X}$
connects the two frontier components of $A_{\min}$ then by the upper
semi-continuity of the decomposition  $\Se(i,j)$ there exists a single component of  $\Se(i,j)$ connecting the two frontier components of $A_{\min}$, which is a contradiction.
\vskip .1 cm
We also claim that $\overline{X} \cap \E (e)(\overline{X})=\emptyset$. Note that $\E(e)(\Se(i,j))=\Se(j,i)$.
Assume  that $\overline{X} \cap \E (e)(\overline{X})\neq \emptyset$ and let $p \in \overline{X} \cap \E (e)(\overline{X})$.
Then $S_p(i,j) \cup S_p(j,i)$ is a closed set that connects the two frontier components of $A_{\min}$. Let $S_1$ and $S_2$ be connected components of
the sets $S_p(i,j)\cap A_{\min}$ and $S_p(j,i)\cap A_{\min}$ respectively, such that $S_1 \cup S_2$ connects the two frontier components of $A_{\min}$.
Let $\wh{S}_1$ and $\wh{S}_2$ be the corresponding single lifts to $P$ such that $\wh{S}_1 \cup \wh{S}_2$ connects the two boundary lines of $P$.
Let $z_0 \in \ph_0(P) \cap \overline{(\wh{S}_1\cup \wh{S}_2)}$ and $z_1 \in \ph_1(P) \cap \overline{(\wh{S}_1\cup \wh{S}_2)}$. Let $\wh{f}$ be a lift of $\E(t_{\gamma})$ to $P$. Then the Euclidean diameter of the set $\wh{f}^{n}(\wh{S}_1 \cup \wh{S}_2)$ is less than $2C_{\min}$, where $C_{\min}$ is the constant defined above. This shows that $|\wh{f}^{n}(z_1)-\wh{f}^{n}(z_0)| \le 2C_{\min}$, for every $n \in \Z$. But this contradicts Proposition \ref {uniform-twist} since $\rho(\E(t_{\gamma}),A_{\min}) \ne 0$.
\vskip .1 cm
Since  $\overline{X} \cap \E (e)(\overline{X})=\emptyset$ we have  that the set $A_{\min}\setminus (\overline{X} \cup\E (e)(\overline{X} ) )$ contains a unique component $A$ that is a topological annulus which is invariant under both $\E(t_{\gamma})$ and $\E (e)$. We will show that $A$ is
an admissible annulus and will contradict that $A_{\min}$ is a minimal annulus.
\vskip .1cm
To show that $A$ is admissible it remains to show that the rotation number of $\E (t_{\gamma})$ on $A$ is an odd integer. There are two cases to consider. The first one is when $\ph{A} \cap \ph{A}_{\min} \ne \emptyset$. Since $A=\E(e)(A)$ we have that there are points $z_i \in \ph{A} \cap \ph_i(A_{\min})$. If $z_i$ are accessible points (with respect to $A$) we can choose an arc $l \subset A$ that connects the two points. Then by
Proposition \ref {intersection-number} we have that $\rho(\E(t_{\gamma}),A_{\min})=\rho(\E(t_{\gamma}),A)$. If $z_i$ is not accessible we can find accessible points that are arbitrary close to $z_i$ and then the argument goes the same way as in the proof of  Lemma \ref{sum-twist-number}. This shows that  $\rho(\E(t_{\gamma}),A_{\min})=\rho(\E(t_{\gamma}),A)$.
\vskip .1cm
If  $\ph{A} \cap \ph{A}_{\min} =\emptyset$ then there are two mutually disjoint annuli $A_i \subset A_{\min}$, $i=0,1$  (and disjoint from $A$) such that
$\ph_i(A_i)=\ph_i(A_{\min})$. Moreover we have that $S_p(i,j)$ connects the two frontier components of $A_0$ and $S_p(j,i)$ connects the two frontier components of $A_1$. Same as above we show that $\rho(\E(t_{\gamma}),A_0)=\rho(\E(t_{\gamma}),A_1)=0$. From Lemma \ref{sum-twist-number} we have  $\rho(\E(t_{\gamma}),A_{\min})=\rho(\E(t_{\gamma}),A)+\rho(\E(t_{\gamma}),A_0)+\rho(\E(t_{\gamma}),A_1)=\rho(\E(t_{\gamma}),A)$. This proves that $A$ is admissible. Therefore $S_p(i,j)$ does not intersect the boundary of $A_{\min}$ and since  $S_p(i,j)$ is a closed set this proves that it is compactly contained in $A_{\min}$.
\end{proof}

\begin{proposition}\label{uniform-diam} Let $p \in A_{\min}$ and $i,j \in \{1,2\}$. Then a single lift of  $S_p(i,j)$ to $P$ has the Euclidean diameter at most $C_{\min}$ ($C_{\min}$ is the constant defined above).
\end{proposition}
\begin{proof} This follows directly from the definition of $C_{\min}$ and the previous proposition.

\end{proof}

\begin{proposition}\label{minimal-idle} Let $K>0$. Then every connected component of the set $\Idle(K,\E(t_{\gamma}),A_{\min})$ is compactly contained in $A_{\min}$
and it does not separate the two frontier components of $A_{\min}$.
\end{proposition}
\begin{proof} The proof of the statement that every connected componet of the set $\Idle(K,\E(t_{\gamma}),A_{\min})$ is compactly contained in $A_{\min}$ is the same as the proof of Proposition \ref{minimal-5} (in fact we have already proved in Proposition \ref{idle-1} that a connected component of
$\Idle(K,\E(t_{\gamma}),A_{\min})$ can not connect the two frontier components of $A_{\min}$). If $Q$ is a connected component of $\Idle(K,\E(t_{\gamma}),A_{\min})$ then we treat the set $\overline{Q} \subset \overline{A}_{\min}$ in the same way as the component $S_p(i,j)$ in the above proof.
\vskip .1cm
Since $A_{\min}$ is a minimal annulus (and as such it does contain another admissible annulus) we conclude from Proposition \ref{idle-2} that a connected component of  $\Idle(K,\E(t_{\gamma}),A_{\min})$  can not separate the two frontier components of $A_{\min}$.
\end{proof}

\begin{proposition}\label{minimal-connected} Let $K>0$. Then  there exists a non-empty connected component $D_{\min}$ of the set $A_{\min} \setminus \Idle(K,\E(t_{\gamma}),A_{\min})$ such that  every connected component of the set
$A_{\min} \setminus D_{\min}$ is compactly contained in $A_{\min}$.
\end{proposition}

\begin{remark} Observe that if $Q$ is connected component of the set $A_{\min} \setminus D_{\min}$ then beside being compactly contained in $A_{\min}$ we have that $Q$ does not separate the two frontier components of $A_{\min}$ because $D_{\min}$ is connected.
Let $\wh{D}_{\min}$ be the lift of $D_{\min}$ to $P$ under the covering map. Then $\wh{D}_{\min}$ is connected. Moreover every connected component of the set $P \setminus \wh{D}_{\min}$ is compactly contained in $P$
\end{remark}

\begin{proof} By the previous propositions  we know that every connected component of the set $\Idle(K,\E(t_{\gamma}),A_{\min})$ is compactly contained in $A_{\min}$ and it does not separate the two frontier components of $A_{\min}$. Therefore we can find an arc

$$
\gamma \subset A_{\min} \setminus \Idle(K,\E(t_{\gamma}),A_{\min}),
$$
\noindent
such that $\gamma$ connects the two frontier components of $A_{\min}$. Let $D_{\min}$ be the connected component of the set $ A_{\min} \setminus \Idle(K,\E(t_{\gamma}),A_{\min})$ that contains $\gamma$. We need to show that every connected component of the set $A_{\min} \setminus D_{\min}$ is compactly contained in $A_{\min}$. Let $Q'$ be such a component and let $Q=Q' \setminus \inte(Q')$. Then $Q \subset \Idle(K,\E(t_{\gamma}),A_{\min})$.
If $Q'$ is not compactly contained in $A_{\min}$ then neither is $Q$. But this is impossible. This proves the proposition.
\end{proof}

\begin{remark} In the previous proof we have that each connected component $Q'$ of the set $A_{\min} \setminus D_{\min}$ is contained in some
$\Idle(K',\E(t_{\gamma}),A_{\min})$, where $K'$ depends on $Q'$.
\end{remark}

\begin{proposition}\label{semi-conjugate} Let $\wt{a}_i \in \Gamma(\ac_i)$ and $\wt{b}_j \in \Gamma(\bc_j)$. Let $\wh{a}_i, \wh{b}_j: \overline{P} \to \overline{P}$ be two lifts. Then there exist $k,l \in \Z$, and  a map $\chi:P \to P$ such that $\chi \circ \wh{a}_i=T^k \circ \chi$, and $\chi \circ \wh{b}_j=T^l \circ \chi$, where $T^k$ and $T^l$ are translations for $k$ and $l$ respectively. Moreover, if $\wh{a}_i$ has a fixed point in $P$ then $k=0$. Similarly if $\wh{b}_j$ has a fixed point in $P$ then $l=0$.
\end{proposition}

\begin{proof} Let $\wh{\Se}(\ac_i,\bc_j)$ be the lift of the minimal decomposition for the group $\Gamma(\ac_i,\bc_j)$ to $P$. Let $\chi:P \to P$ be the Moore's map, that is $\chi$ maps every component of $\wh{\Se}(\ac_i,\bc_j)$ to a point. Then $\chi$ is the required map. If $\wh{a}_i(p)=p$, for some $p \in P$, then  $\chi(p)=T^k(\chi(p))$ which shows that $k=0$. The same argument goes if  $\wh{b}_j$ has a fixed point in $P$.
\end{proof}

\section{The proof of Theorem \ref{main} }

\subsection{Special subsets of the minimal annulus} So far we have considered the minimal decompositions $\Se(i,j)$ for the groups $\Gamma(\ac_i,\bc_j)$, $i,j \in \{1,2\}$. Let $\Gamma(\ac_i)$, $i=1,2$,  be the group genererated by all $\E(t_{\alpha})$,  where $\alpha \in \ac_i$. Let  $\Gamma(\bc_i)$, $i=1,2$,  be the group genererated by all $\E(t_{\beta})$, where $\beta \in \bc_i$. Let $\Se(\ac_i)$ and $\Se(\bc_j)$ be the corresponding minimal decompositions. Then for $p \in A_{\min}$ we have that $S_p(\ac_i) \subset S_p(i,j)$, and $S_p(\bc_j) \subset S_p(i,j)$. This implies that every such component $S_p(\ac_i)$ (or
$S_p(\bc_j) $) is compactly contained in $A_{\min}$ and a single lift of  $S_p(\ac_i)$ (or $S_p(\bc_j) $) to the strip  $P$ has the Euclidean diameter less than $C_{\min}$.

\begin{definition} Let $p \in A_{\min}$. We say that $p \in O_0$ if for some pair $i,j \in \{1,2\}$, we have that $p$ belongs to the interior of the  component $S_p(i,j) \in \Se(i,j)$. We say that $p \in E$ if $\E(t_{\gamma})$ setwise fixes at least one of the four components $S_p(\ac_1) \in \Se(\ac_1)$,  $S_p(\ac_2) \in \Se(\ac_2)$,  $S_p(\bc_1) \in \Se(\bc_1)$,   $S_p(\bc_2) \in \Se(\bc_2)$.  Set $X=A_{\min} \setminus (O_0 \cup E)$.
\end{definition}

\begin{definition}
Let $i,j\in\{ 1,2\}$ be fixed. Define $X_{\ac_i,\bc_j}$ to be
the set of all $p \in X$ such that every element of the group $\Gamma(\ac_i)$
setwise fixes the component $S_p(\bc_j) \in \Se(\bc_j)$. Define $X_{\bc_i,\ac_j}$ to be
the set of all $p \in X$ such that every element of the group $\Gamma(\bc_i)$
setwise fixes the component $S_p(\ac_j) \in \Se(\ac_j)$.
\end{definition}

\begin{proposition}\label{X-0} For every pair $i,j \in \{1,2\}$ we have
\begin{equation}\label{X}
X=X_{\ac_i,\bc_j} \bigcup X_{\bc_j,\ac_i}.
\end{equation}
\end{proposition}

\begin{proof} If $p \in X$ then $p$ does not belong to the set $O_0$. The identity (\ref{X}) then follows directly from Lemma \ref{triod-lemma}.
\end{proof}

Now we use the Artin type relations introduced in Section 2.

\begin{proposition}\label{X-1} We have $X_{\ac_1,\bc_i} \cap X_{\ac_2,\bc_i}=\emptyset$ and  $X_{\bc_1,\ac_i} \cap X_{\bc_2,\ac_i}=\emptyset$, for every $i \in \{1,2\}$.
\end{proposition}
\begin{proof} We show  $X_{\ac_1,\bc_1} \cap X_{\ac_2,\bc_1}=\emptyset$. The other case is proved in the same way.
Assume that $X_{\ac_1,\bc_1} \cap X_{\ac_2,\bc_1}\ne \emptyset$. Let  $p \in  X_{\ac_1,\bc_1} \cap X_{\ac_2,\bc_1}$.
Then all elements of both  groups $\Gamma(\ac_1)$ and  $\Gamma(\ac_2)$ setwise fix the set
$S_p(\bc_1)$. We apply the Artin type relation (when the genus of $M$ is even we apply (\ref{even-relation}) and when genus is odd we apply (\ref{odd-relation})) and obtain that $\E(t_{\gamma})$ setwises fixes the set $S_p(\bc_1)$. This shows that $p \in E$ which contradicts the assumption
$p \in X$ since $X \cap E=\emptyset$.
\end{proof}

\begin{proposition}\label{X-2} We have $X_{\ac_i,\bc_j} \cap X_{\bc_j,\ac_i}=\emptyset$ for every pair $i,j \in \{1,2\}$.
\end{proposition}
\begin{proof}
Assume that $ X_{\ac_1,\bc_1} \cap X_{\bc_1,\ac_1} \ne \emptyset$. We derive a contradiction.  Let $p \in  X_{\ac_1,\bc_1} \cap X_{\bc_1,\ac_2}$. Then
every element of the group $\Gamma(\ac_1)$ setwise fixes the component $S_p(\bc_1) \in \Se(\bc_1)$ and  every element of the group $\Gamma(\bc_1)$
setwise fixes the component $S_p(\ac_1) \in \Se(\ac_1)$.  Now we apply (\ref{X}) to the pair $(i,j)=(1,2)$. This shows that at least one of the following holds:
\begin{enumerate}
\item   Every element of the group $\Gamma(\ac_1)$ setwise fixes the component $S_p(\bc_2)$.
\item Every element of the group $\Gamma(\bc_2)$ setwise fixes the component $S_p(\ac_1)$.
\end{enumerate}

Assume that $(2)$ holds. Then  the conclusion is that all elements of both  groups $\Gamma(\bc_1)$ and  $\Gamma(\bc_2)$ setwise fix the set
$S_p(\ac_1)$. We then apply the Artin type relation (when the genus of $M$ is even we apply (\ref{even-relation}) and when genus is odd we apply (\ref{odd-relation})) and obtain that $\E(t_{\gamma})$ setwises fixes the component $S_p(\ac_1)$. But then $p \in E$ which contradits that $p \in X$.
This shows that the $(2)$ can not hold so we conclude that  every element of the group $\Gamma(\ac_1)$ setwise fixes the component $S_p(\bc_2)$. If we apply (\ref{X}) to the pair $(i,j)=(2,1)$, by the same argument we conclude that  every element of the group $\Gamma(\bc_1)$ setwise fixes the component $S_p(\ac_2)$. Let us collect the statements we have proved so far
\begin{enumerate}
\item Every element of the group $\Gamma(\ac_1)$ setwise fixes the component $S_p(\bc_1)$.
\item Every element of the group $\Gamma(\bc_1)$ setwise fixes the component $S_p(\ac_1)$.
\item Every element of the group $\Gamma(\ac_1)$ setwise fixes the component $S_p(\bc_2)$.
\item Every element of the group $\Gamma(\bc_1)$ setwise fixes the component $S_p(\ac_2)$.
\end{enumerate}

We now apply  (\ref{X}) to the pair $(i,j)=(2,2)$. If every element of the group $\Gamma(\ac_2)$ setwise fixes the component $S_p(\bc_2)$ then we have that  all elements of both  groups $\Gamma(\ac_1)$ and  $\Gamma(\ac_2)$ setwise fix the set
$S_p(\bc_2)$. Again by the corresponding Artin type relation this implies that  $\E(t_{\gamma})$ setwises fixes the component $S_p(\bc_2)$. This shows that $p \in E$ which is a contradiction. Similarly if every element of the group $\Gamma(\bc_2)$ setwise fixes the component $S_p(\ac_2)$ then
all elements of both  groups $\Gamma(\bc_1)$ and  $\Gamma(\bc_2)$ setwise fix the set $S_p(\ac_2)$. Again a contradiction.
\vskip .1cm
One similarly shows that  $ X_{\ac_i,\bc_j} \cap X_{\bc_j,\ac_i}= \emptyset$ for other three pairs $(i,j)$.
\end{proof}

\begin{lemma} \label{X-3} We have
$$
X_1=X_{\bc_1,\ac_1}=X_{\ac_1,\bc_2}=X_{\bc_2,\ac_2}=X_{\ac_2,\bc_1},
$$
and
$$
X_2=X_{\bc_2,\ac_1}=X_{\ac_1,\bc_1}=X_{\bc_1,\ac_2}=X_{\ac_2,\bc_2}.
$$
\noindent
The sets $X_1$ and $X_2$ are disjoint.
\end{lemma}

\begin{proof} The first two identities follow  from the previous three propositions. This is seen as follows. We have
$X_{\bc_1,\ac_1} \cup X_{\ac_1,\bc_1}=X$ and this union is disjoint.  Since $X_{\bc_1,\ac_1} \cap X_{\bc_2,\ac_1}=\emptyset$ we have
$$
X_{\bc_2,\ac_1} \subset X_{\ac_1,\bc_1}.
$$
\noindent
Since $X_{\ac_1,\bc_1} \cap X_{\ac_2,\bc_1}=\emptyset$ we have
$$
X_{\ac_2,\bc_1} \subset X_{\bc_1,\ac_1}.
$$
Next, we have  $X_{\bc_2,\ac_1} \cup X_{\ac_1,\bc_2}=X$ and this union is disjoint.  Since $X_{\bc_2,\ac_1} \cap X_{\bc_1,\ac_1}=\emptyset$ we have
$$
X_{\bc_1,\ac_1} \subset X_{\ac_1,\bc_2}.
$$
\noindent
Since $X_{\ac_1,\bc_2} \cap X_{\ac_2,\bc_2}=\emptyset$ we have
$$
X_{\ac_2,\bc_2} \subset X_{\bc_2,\ac_1}.
$$

Next we have  $X_{\bc_1,\ac_2} \cup X_{\ac_2,\bc_1}=X$ and this union is disjoint.  Since $X_{\bc_1,\ac_2} \cap X_{\bc_2,\ac_2}=\emptyset$ we have
$$
X_{\bc_2,\ac_2} \subset X_{\ac_2,\bc_1}.
$$
\noindent
Since $ X_{\ac_2,\bc_1} \cap  X_{\ac_1,\bc_1}=\emptyset$ we have
$$
X_{\ac_1,\bc_1} \subset X_{\bc_1,\ac_2}.
$$
\noindent
Finally we have $X_{\bc_2,\ac_2} \cup X_{\ac_2,\bc_2}=X$. In the same way as above this shows
$$
X_{\bc_1,\ac_2} \subset X_{\ac_2,\bc_2},
$$
\noindent
and
$$
X_{\ac_1,\bc_2} \subset X_{\bc_2,\ac_2}.
$$
\noindent
Combining these eight inclusions we obtain the first two identities of this proposition.
\vskip .1cm
The sets $X_1$ and $X_2$ are disjoint because $X_1=X_{\bc_1,\ac_1}$ and $X_2=X_{\ac_1,\bc_1}$ and by Proposition \ref{X-2} we have that
$X_{\bc_1,\ac_1}$ and  $X_{\ac_1,\bc_1}$ are disjoint.
\end{proof}

\subsection{The proof of Theorem \ref{main}} Let $\Theta< \Homeo(M)$ be the group generated by the elements from all four groups $\Gamma(\ac_i,\bc_j)$.
\begin{definition} Let
$$
O=\bigcup_{\wt{\theta} \in \Theta} \wt{\theta}(O_0).
$$
\end{definition}
We have

\begin{proposition}\label{minimal-long}  Let $\wh{O}$ be the  lift of $O$ to $P$ and let $\wh{f}:\overline{P} \to \overline{P}$ be a lift of $\E(t_{\gamma})$ to $\overline{P}$. Assume that the points $p$ and $q$ belong to the same connected component of the set  $\wh{O}$. Then the cyclic group generated by $\wh{f}$ is $K$ long range Lipschitz on the pair of points $p,q$, for some $K>0$ (the constant $K$  may depend on the choice of $p$ and $q$)
\end{proposition}

\begin{proof}  Every homeomorphism from $\Theta$ setwise preserves the minimal annulus $A_{\min}$.
Let $\wt{\theta} \in \Theta$ and let $\wh{\theta}:\overline{P} \to \overline{P}$ be a lift. Let $K(\wt{\theta})$ be defined so that for every  two points $z_1,z_2 \in \overline{P}$, such that $|z_1-z_2| \le C_{\min}$, we have that $|\wh{\theta}(z_1)-\wh{\theta}(z_2)| <K(\wt{\theta})$. The constant $K(\wt{\theta})$ exists because $\wh{\theta}$ is uniformly continuous on $\overline{P}$ (it commutes with the translation for $1$), and $K(\wt{\theta})$  does not depend on the choice of a lift $\wh{\theta}$.
\vskip .1cm
Fix $\wt{\theta} \in \Theta$. Every point  $r \in O_0$ is contained in the interior of the component $S_r(i,j) \in \Se(i,j)$, for some pair $i,j$. A single lift of $S_r(i,j)$ to $P$ has the Euclidean diameter less than $C_{\min}$. Let $S=\wt{\theta}(S_r(i,j))$ and denote by $\wh{S}$ a single lift of $S$ to $P$.
Then the Euclidean diameter of $\wh{S}$ is less than $K(\wt{\theta})$.  The homeomorphism $\E(t_{\gamma})$ commutes with $\wt{\theta}$ (also recall that $\E(t_{\gamma})$ respects the decomposition $\Se(i,j)$). This implies  that the set $\wh{f}^{n}(\wh{S})$ is a single lift of the interior of some $\wt{\theta}(S_q(k,l))$. This shows that the Euclidean diameter of the set $\wh{f}^{n}(\wh{S})$ is less than $K(\wt{\theta})$.
\vskip .1cm
Since $p$ and $q$ are in the same connected component of $\wh{O}$ there exists an arc $l \subset \wh{O}$ with the endpoints $p$ and $q$. Every point on $l$ belongs to a single lift of some component $\wt{\theta}(S_r(i,j))$, for some $\wt{\theta}\in \Theta$, some $r \in O_0$ and some pair $i,j$.
Since $l$ is a closed subset of $\wh{O}$ we can find finitely many such components $\wt{\theta}(S_r(i,j))$. Therefore the arc $l$ is covered by finitely many sets $D_1,...,D_k \subset P$, where each $D_t$, $t=1,..,k$, is a single lift of some $\wt{\theta}(S_r(i,j))$ to $P$. Therefore there exists a constant $K'>0$ such that the Euclidean diameter of every set $\wh{f}^{n}(D_i)$, $n \in \N$,  is less than $K'$. This shows that $|\wh{f}^{n}(p)-\wh{f}^{n}(q)| \le kK'$, for every $n \in \N$. Set $K=kK'$. This proves the proposition.
\end{proof}

\begin{proposition}
We have $(A_{\min} \setminus O) \subset \Idle(3C_{\min},\E(t_{\gamma}),A_{\min})$.
\end{proposition}
\begin{proof} If $p \in (A_{\min} \setminus O)$ then either $p \in E$ or $p \in X$ (in fact this is true by definition for every
$p \in (A_{\min} \setminus O_0)$). Assume that $p \in E$. Then there exists a component $S_p(i,j)$ such that $\E(t_{\gamma})(S_p(i,j))=S_p(i,j)$.
Let $\wh{S}_p(i,j)$ be a single lift of $S_p(i,j)$ to $P$ and let $\wh{p} \in \wh{S}_p(i,j)$ be the corresponding  lift of $p$ to $P$.
Moreover let $\wh{f}$ be the lift of $\E(t_{\gamma})$ to $P$ that setwise fixes $\wh{S}_p(i,j)$. We have that the Euclidean diameter of $\wh{S}_p(i,j)$ is less than $C_{\min}$ and $\wh{f}(\wh{S}_p(i,j))=\wh{S}_p(i,j)$. This shows that $|\wh{f}^{n}(\wh{p})-\wh{p}| \le C_{\min}$, for every $n \in \Z$. Therefore $p \in  \Idle(2C_{\min},\E(t_{\gamma}),A_{\min})$.
\vskip .1cm
Assume that $p \in X$. We may assume that $p \in X_1$ (the case  $p \in X_2$ is treated in the same way).
Let $S_p(\ac_1) \in \Se(\ac_1)$ be the corresponding component that contains $p$. There are two cases to consider.
\vskip .1cm
The first case is when  $S_p(\ac_1) \cap E \ne \emptyset$. Let $q \in S_p(\ac_1) \cap E \ne \emptyset$. Let $S_q(*)$ be the component such that $\E(t_{\gamma})(S_q(*))=S_q(*)$. This implies that
$$
\E(t_{\gamma})^{n}(S_p(\ac_1)) \cap S_q(*)  \ne \emptyset,
$$
\noindent
for every $n \in \Z$.
Let $\wh{q}$ be a lift of $q$ to $P$ and denote by $\wh{S}_{q}(*)$ the corresponding lift of  $S_q(*)$  to $P$ that contains $\wh{q}$.
Let $\wh{S}_{p}(\ac_1)$ be the corresponding lift of $S_p(\ac_1)$ to $P$ such that $\wh{q} \in \wh{S}_p(\ac_1)$. Also let
$\wh{p}$ be the lift of $p$ to $P$ such that $\wh{p} \in \wh{S}_{p}(\ac_1)$. Finally let $\wh{f}$ be the lift of $\E(t_{\gamma})$ to $P$ that setwise fixes the set $\wh{S}_{q}(*)$.
\vskip .1cm
Then for every $n \in \Z$ we have that $\wh{f}^{n}(\wh{p})$ belongs to the set $\wh{f}^{n}(\wh{S}_{p}(\ac_1))$, and $\wh{f}^{n}(\wh{S}_{p}(\ac_1))$ has non-empty intersection with $\wh{S}_{q}(*)$. Since the Euclidean diameter of each of the sets $\wh{S}_{q}(i,j)$ and $\wh{f}^{n}(\wh{S}_{p}(\ac_1))$, for every $n \in \Z$ is less than $C_{\min}$, we have that
$|\wh{f}^{n}(\wh{p})- \wh{f}^{m}(\wh{p})| \le 3C_{\min}$, for every $m, \, n \in \Z$ (that is for every $m,n \in \Z$ we have that the points
$\wh{f}^{n}(\wh{p})$ and $\wh{f}^{m}(\wh{p})$ are contained in the union of three sets, and  each of these three sets has the Euclidean diameter at most $C_{\min}$).
This shows that $p \in  \Idle(3C_{\min},\E(t_{\gamma}),A_{\min})$.
\vskip .1cm
The second case is when $S_p(\ac_1) \cap E =\emptyset$. Let $q \in S_p(\ac_1)$ be any point such that $q$ does not belong to $O$. Then either $q \in X_1$ or $q \in X_2$, and not both can hold since $X_1 \cap X_2=\emptyset$.  Since $p \in X_1$ we have that $\wt{b}_1(S_p(\ac_1))=S_p(\ac_1)$, for every $\wt{b}_1 \in \Gamma(\bc_1)$. If $q \in X_2$ then by Lemma \ref{X-3} we would have $\wt{b}_1(S_q(\ac_1)) \ne  S_q(\ac_1)$, for some $\wt{b}_1 \in \Gamma(\bc_1)$. But $S_p(\ac_1)=S_q(\ac_1)$, so we find that $q \in X_1$. This implies that $\wt{a}_2(S_q(\bc_1))=S_q(\bc_1)$,  for every $\wt{a}_2 \in \Gamma(\ac_2)$, and for every $q \in (S_p(\ac_1) \setminus O)$.
\vskip .1cm
Let $\wh{p}$ be a lift of $p$ to $P$ and let $\wh{S}_{p}(\ac_1)$ be the single lift of $S_p(\ac_1)$ to $P$ such that $\wh{p} \in \wh{S}_{p}(\ac_1)$. Let $\wh{S}_{p}(\bc_1)$ be the lift of $S_p(\bc_1)$ to $P$ that contains $\wh{p}$. For $\wt{a}_1 \in \Gamma(\ac_1)$ let $\wh{a}_1$ be the lift to $P$ that setwise fixes $\wh{S}_{p}(\ac_1)$. Also for every $\wt{b}_1 \in \Gamma(\bc_1)$ let $\wh{b}_1$ be the lift to $P$ that setwise fixes $\wh{S}_{p}(\bc_1)$ (we do not need this fact, but it can be easily  shown that the lift $\wh{b}_1$ setwise fixes the set $\wh{S}_{p}(\ac_1)$ as well). For $\wt{a}_2 \in \Gamma(\ac_2)$ let $\wh{a}_2$ be the lift to $P$ that setwise fixes $\wh{S}_{p}(\bc_1)$.
\begin{remark} Observe that each $\wh{a}_2$ and $\wh{b}_1$ has a fixed point in $P$. This follows from the assumption that $\wh{a}_2(\wh{S}_p(\bc_1))=
\wh{b}_1(\wh{S}_p(\bc_1))=\wh{S}_p(\bc_1)$, and since $\wh{S}_p(\bc_1)$ is an acyclic set we see that each $\wh{a}_2$ and $\wh{b}_1$ has a fixed point in
$\wh{S}_p(\bc_1)$.
\end{remark}

Next we want to show that every $\wh{a}_2 \in \Gamma(\ac_2)$ setwise preserves the lift of every component from $\Se(\bc_1)$ that intersects $\wh{S}_{p}(\ac_1)$ in some point that does not belong to $\wh{O}$ (where $\wh{O}$ is the total lift of $O$ to $P$). Fix $\wh{a}_2$. Let $\wh{S}_q(\bc_1)$ be the lift of a component from $\Se(\bc_1)$ such that $\wh{S}_q(\bc_1)$ intersects $\wh{S}_{p}(\ac_1)$, and such that this intersection contains a point $\wh{q}$ that is not in $\wh{O}$. Then  $\wh{q} \in \wh{X}_1$, where $\wh{X}_1$ is the total lift of $X_1$ to $P$.  This yields the relation $\wh{a}_2(\wh{S}_q(\bc_1))=\wh{S}_q(\bc_1)+k$, for some $k \in \Z$. Let $\wh{S}_q(2,1)$ be the lift of the corresponding component from $\Se(2,1)$ that contains $q$.  Since  each $\wh{a}_2$ and $\wh{b}_1$ has a fixed point in $P$, from Proposition \ref{semi-conjugate} we see that $\wh{a}_2(\wh{S}_q(2,1))=\wh{b}_1(\wh{S}_q(2,1))=\wh{S}_q(2,1)$. But then $\wh{S}_q(2,1)$ contains both $\wh{S}_q(\bc_1)$ and $\wh{S}_q(\bc_1)+k$. Since $S_q(2,1)$ is acyclic we see that $k=0$ that is $\wh{a}_2$ setwise fixes $\wh{S}_q(\bc_1)$.
\vskip .1cm
Let $\Theta(\ac)$ be the group generated by the elements from $\Gamma(\ac_1)$ and $\Gamma(\ac_2)$.
Let $\wt{\theta} \in \Theta(\ac)$. Then $\wt{\theta}=\wt{a}^{1}_1 \circ \wt{a}^{1}_2 \circ \wt{a}^{2}_1 \circ \wt{a}^{2}_2 \circ... \circ \wt{a}^{k}_1 \circ \wt{a}^{k}_2$, for some $\wt{a}^{j}_i \in \Gamma(\ac_i)$, where $i=1,2$. By appropriately choosing the lifts of each  $\wt{a}^{j}_i$ we see
that there exists a lift  $\wh{\theta}:P \to P$ of $\wt{\theta}$ such that $\wh{\theta}(\wh{p})$ belongs to a lift of some component from $\Se(\bc_1)$ to $P$ that intersects $\wh{S}_p(\ac_1)$, and this intersection contains a point that is not in $\wh{O}$ (here we use that $\wt{\theta}(O)=O$). Since $\E(t_{\gamma}) \in \Theta(\ac)$ we have that $p \in \Idle(2C_{\min},\E(t_{\gamma}),A_{\min})$ (that is for every $m,n \in \Z$ we have shown that the points  $\wh{f}^{n}(\wh{p})$ and $\wh{f}^{m}(\wh{p})$ are contained in the union of two sets, and  each of these two sets has the Euclidean diameter at most $C_{\min}$).
\end{proof}
\vskip .1cm
Now we are ready to prove Theorem \ref{main}. Let $K=3C_{\min}$. Let $D_{\min}$ be the corresponding connected component of the set
$A_{\min} \setminus \Idle(K,\E(t_{\gamma}),A_{\min})$ from Proposition \ref{minimal-connected}.
By the previous proposition we have that
$$
D_{\min} \subset A_{\min} \setminus \Idle(K,\E(t_{\gamma}),A_{\min}) \subset O.
$$
\noindent
On the other hand by Proposition \ref{minimal-long} we have that if $\wh{f}$ denotes a lift of $\E(t_{\gamma})$ to $P$,
then the cyclic group genereted by $\wh{f}$ is $N$ long range Lipschitz on the pair of points $p,q$ in $D_{\min}$, for some $N>0$ (the constant $N$  may depend on the choice of $p$ and $q$). Therefore we may apply  Lemma \ref{long-range}. This lemma implies that $\rho(\E(t_{\gamma}),A_{\min})=0$. But this is a contradiction.


\begin{thebibliography}{99}

\bibitem{c-c} S. Cantat, D. Cerveau, \textsl{Analytic actions of mapping class groups on surfaces}. available at
http://perso.univ-rennes1.fr/serge.cantat/ ,Preprint


\bibitem{f-m} B. Farb and D. Margalit, \textsl{ A Primer on Mapping Class Groups}. Preprint


\bibitem{f-h-1} J. Franks, M. Handel, \textsl{Complete semi-conjugacies for psuedo-Anosov homeomorphisms}. arXiv:0712.3069, Preprint

\bibitem{f-h-2} J. Franks, M. Handel, \textsl{Global fixed points for centralizers and Morita's Theorem}. arXiv:0801.0736, Preprint


\bibitem{iv} N. Ivanov, \textsl{Mapping Class Groups}. Preprint

\bibitem{l-v}  O. Lehto, K. Virtanen, \textsl{Quasiconformal mappings in the plane}. Die Grundlehren der mathematischen Wissenschaften, Band 126. Springer-Verlag, New York-Heidelberg, (1973)


\bibitem{ma} V. Markovic, \textsl{Realization of the mapping class group by homeomorphisms}.  Inventiones Mathematicae  168 ,  no. 3, 523--566 (2007)


\bibitem{mo} R. Moore, \textsl{Concerning triods in the plane and the junction points of the plane continua}. Proc. Nat. Acad. Sci. U. S. A. 14, 85-88 (1928)


\bibitem{mo-1} R. Moore, \textsl {Foundations of point set theory}. Revised edition. American Mathematical Society Colloquium Publications, Vol. XIII American Mathematical Society, Providence, R.I. (1962)


\bibitem{mor} S. Morita, \textsl{Characteristic classes of surface bundles}.  Inventiones Mathematicae  90,  no. 3, 551--577. (1987)


\bibitem{po} Ch. Pommerenke, \textsl{Boundary behaviour of conformal maps}. Grundlehren der Mathematischen Wissenschaften, 299. Springer-Verlag, Berlin, (1992)




\end{thebibliography}
\end{document}